\documentclass[10pt]{amsart}
\usepackage{amsmath,amssymb,amsthm,amsfonts,mathrsfs,amsopn}
\usepackage[all]{xy}
\usepackage{dsfont}
\usepackage{color}
\usepackage{esint}
\usepackage{mathtools}
\mathtoolsset{showonlyrefs}
\usepackage{slashed}

\usepackage{xfrac}
\usepackage{faktor}
  
\usepackage{tikz-cd}

\usepackage[margin=1.1in]{geometry}

\newcommand{\p}{\partial}
\newcommand{\R}{\mathbb{R}}

\newcommand{\cliff}{\mathfrak{m}}
\newcommand{\D}{\slashed{D}}

\newcommand{\dd}{\mathop{}\!\mathrm{d}}
\newcommand{\eps}{\varepsilon}
\newcommand{\snabla}{\slashed{\nabla}} 
\newcommand{\sph}{\mathbb{S}}

\newcommand{\Abracket}[1]{\left<#1\right>} 
\newcommand{\parenthesis}[1]{\left(#1\right)} 
\newcommand{\braces}[1]{\left\{#1\right\}} 

\newcommand{\bB}{\mathbb{B}}

\newcommand{\bZ}{\mathbb{Z}}
\newcommand{\cC}{\mathcal{C}}
\newcommand{\cD}{\mathcal{D}}

\newcommand{\sL}{\mathscr{L}}

\newcommand{\sN}{\mathscr{N}}

\DeclareMathOperator{\dv}{\dd{vol}}

\DeclareMathOperator{\Eigen}{Eigen}
\DeclareMathOperator{\End}{End}

\DeclareMathOperator{\Hess}{Hess}

\DeclareMathOperator{\id}{Id}

\DeclareMathOperator{\Spect}{Spect}

\DeclareMathOperator{\Spin}{Spin}

\DeclareMathOperator{\Vol}{Vol}

\newtheorem{thm}{Theorem}[section]

\newtheorem{lemma}[thm]{Lemma}
\newtheorem{prop}[thm]{Proposition}
\newtheorem{cor}[thm]{Corollary}

\newtheorem{rmk}[thm]{Remark}

\title[Super sinh-Gordon equations]{Min-max solutions for Super sinh-Gordon equations \\on compact surfaces}

\author[A. Jevnikar]{Aleks Jevnikar}
\address{Aleks Jevnikar, Department of Mathematics, Computer Science and Physics, University of Udine, Via delle Scienze 206, 33100 Udine, Italy.}
\email{aleks.jevnikar@uniud.it}

\author[A. Malchiodi]{Andrea Malchiodi}
\address{Andrea Malchiodi, Scuola Normale Superiore, Piazza dei Cavalieri 7, 56126 Pisa, Italy}
\email{andrea.malchiodi@sns.it}

\author[R. Wu]{Ruijun Wu}
\address{Ruijun Wu, SISSA, Via Bonomea, 265, 34136 Trieste, Italy}
\email{ruijun.wu@sissa.it}

\begin{document}

\begin{abstract}
In the present paper we initiate the variational analysis of a super sinh-Gordon system on compact surfaces, yielding the first example of non-trivial solution of min-max type. The proof is based on a linking argument jointly with a suitably defined Nehari manifold and a careful analysis of Palais-Smale sequences. We complement this study with a multiplicity result exploiting the symmetry of the problem.  
\end{abstract}

\maketitle

{\footnotesize
\emph{Keywords}: super sinh-Gordon equations, existence results, min-max methods, multiplicity results.

\medskip

\emph{2010 MSC}: 58J05, 35A01, 58E05, 81Q60.}

\

\section{Introduction}

A general Lagrangian in supersymmetric string theory can be expressed as
\begin{align}\label{eq:Lagrangian-physics}
 \sL(v,\phi)
 =& \frac{1}{8\pi}|\nabla v|^2 - \frac{1}{\pi}\Abracket{\D\phi,\phi}
   -\frac{1}{4\pi}|\phi|^2 W''(v) + \frac{1}{32\pi}[W'(v)]^2,
\end{align}
see~\cite{ahn2002RG, pensot2003massless}, where~$W=W(v)$ is a superpotential,~$v$ is a scalar function and~$\D$ is the Dirac operator acting on spinors~$\phi$, see Section~\ref{sec:preliminaries} for precise definitions. For instance, the super Liouville equations, which have  attracted a great attention~\cite{super1,super2,jevnikar2020existence,jevnikar2020existence-sphere,jost1,jost2,jost3,jost4,super3}, are recovered by considering the potential~$W(v)= 8\pi \mu e^{bv}$. Here we will be concerned with the super sinh-Gordon equations (SShG) with potential 
\begin{align}\label{pot}
 W(v)=8\pi \mu\cosh(bv),
\end{align}
where~$\mu>0$ and~$b>0$ are physical parameters, see~\cite{ahn2002RG, pensot2003massless}. This model can be seen as a perturbation of the super Liouville equations with spontaneously broken supersymmetry, where the massless fermions plays the role of the goldstino. Under a suitable transformation, the model is equivalent to its imaginary coupling version: the super sine-Gordon equations (SSG) with ~$W(v)=8\pi\mu\cos(bv)$. Both (SSG) and (SShG) can be mapped into an affine Toda theory based on the twisted super Lie algebra~$C^{(2)}(2)$~\cite{chaichian1987superconformal}. By the choice made in~\eqref{pot}, we have
\begin{align}
 &W'(v)= 8\pi\mu b\sinh(bv), \\ 
 &W''(v)=8\pi\mu b^2\cosh(bv),
\end{align}
and hence 
\begin{align}
 \sL(v,\phi)
 =\frac{1}{8\pi}|\nabla v|^2-\frac{1}{\pi}\Abracket{\D\phi,\phi}
 -2\mu b^2\cosh(bv)|\phi|^2 
 +2\pi\mu^2 b^2 \sinh(bv)^2.
\end{align}
Some simplification occurs by introducing the variables 
\begin{align}
 &u\coloneqq bv, \\  
&\psi\coloneqq b\omega\cdot\phi
\end{align}
where~$\omega=\cliff(e_1)\cliff(e_2)$ denotes the Clifford multiplication of the two-dimensional real volume element (see Section~\ref{sec:preliminaries}). We have then 
\begin{align}
 &|\nabla u|^2 = b^2|\nabla v|^2, \quad |\psi|^2= b^2|\phi|^2, \\
 &\D\psi=b\D(\omega\cdot\phi)=-b \omega\cdot \D\phi, \quad 
 \Abracket{\D\psi,\psi}=- b^2\Abracket{\D\phi,\phi},
\end{align}
and 
\begin{align}
 8\pi b^2 \sL
 = |\nabla u|^2 +8\Abracket{\D\psi,\psi}-16\pi\mu b^2\cosh(u)|\psi|^2 +16\pi^2\mu^2 b^4 \sinh(u)^2. 
\end{align}
With~$\rho\equiv 2\pi\mu b^2\in \R$, we will consider the Lagrangian density
\begin{align}
 L_\rho(u,\psi)
 =|\nabla u|^2+8\Abracket{\D\psi,\psi}-8\rho\cosh(u)|\psi|^2 
 + 4\rho^2 \sinh(u)^2. 
\end{align}
The aim of this work is to obtain the existence of solutions of the corresponding Euler-Lagrange equations via variational methods.

\

From now on, let~$M$ be a compact Riemann surface with empty boundary and endowed with a smooth Riemannian metric~$g$. 
Fix a spin structure and let~$\Sigma M\equiv \Sigma_g M$ denote the associated spinor bundle over~$M$. 
We will consider the functional 
\begin{align}
 J_\rho \colon H^1(M)\times H^{\frac{1}{2}}(\Sigma M) \to \R
\end{align}
defined by the above Lagrangian~$L_\rho$, namely
\begin{align}\label{eq:action functional}
 J_\rho(u,\psi)
 = \int_M \left(  |\nabla u|^2+8\Abracket{\D\psi,\psi}-8\rho\cosh(u)|\psi|^2 
 + 4\rho^2 \sinh(u)^2 \right) \dv_g. 
\end{align}
The Euler--Lagrange equations are given by
\begin{equation}\label{eq:SShG-intro} \tag{SShG}
 \begin{cases}
  \Delta u= 2\rho^2 \sinh(2u)- 4\rho \sinh(u)|\psi|^2, \\
  \D\psi= \rho \cosh(u)\psi. 
 \end{cases}
\end{equation}
It is clear that any weak solution~$(u,\psi)\in H^1(M)\times H^{\frac{1}{2}}(\Sigma M)$ is indeed smooth, thanks to the Moser-Trudinger embedding and the regularity theory for Laplacian as well as Dirac's operators, see e.g.~\cite{ammann2003habilitation,aubin1998somenonlinear, jost2018regularity}. 
Here we are concerned with the existence issue of weak solutions of~\eqref{eq:SShG-intro}.

There are some easy solutions for~\eqref{eq:SShG-intro} which are essentially trivial. Clearly~$(0,0)$ is a solution. 
Moreover, if~$u=0$ and~$\rho=\lambda_k\in\Spect(\D_g)$, then any~$\psi\in \Eigen(\D_g;\lambda_k)$ gives a nonzero solution. 
In the general case, we prove the following existence result in the present work. 

\begin{thm}\label{thm:main thm}
Let~$M$ be a compact Riemann surface and suppose~$\rho\notin \Spect(\D)$. Then, the~\eqref{eq:SShG-intro} admits a nonzero solution. Moreover, if the curvature~$K(g)$ is such that~$K(g)\le c<0$, then there exists a nontrivial solution~$(u,\psi)$ with~$u\not\equiv const.$ 
\end{thm}

At least to our knowledge, this seems to be the first example of non-trivial solution for the~\eqref{eq:SShG-intro} in this setting.

Note that neither the action functional~$J_\rho$ in~\eqref{eq:action functional} nor the super sinh-Gordon equations~\eqref{eq:SShG-intro} is conformally invariant, which is in great contrast with (super) Liouville equations. 
This aspect will be further taken into account in the blowup analysis of the system~\eqref{eq:SShG-intro}, which will be addressed in another work. 

Let us comment on the last statement in the above theorem. When there exist eigenspinors of constant length, we can again find solutions with nonzero but constant function components.
However, by a classical result of Friedrich~\cite[Theorem 13]{friedrich1998spinorrepresentation}, if there exists a nonzero eigenspinor of constant length, then it induces a constant mean curvature (CMC) immersion of the universal cover~$\widetilde{M}$ of~$(M,g)$, which is impossible by Efimov theorem~\cite{efimov1964generation, efimove1966hyperbolic,milnor1972efimov} if the compact surface~$M$ has negative curvature~$K(g)\le c<0$. 
Thus in such a case, if we obtain a nonzero solution, then we are sure that it is a nontrivial solution. 
On the other hand, in case~$K(g)$ changes sign, if we obtain constant length solutions, they would induce CMC immersions of the universal cover, which are also interesting. See  the end of Section~\ref{sec:preliminaries} for more details in this respect.
 
\

These types of problems were studied for a long time, but the existence theory for Dirac's operator is far from satisfactory, especially when coupled in a 
system of equations. 
However, the super extensions of the classical equations are quite natural in physics, as they are the basic equations describing fermionic fields and thus deserve a thorough mathematical theory.
Here we address~\eqref{eq:SShG-intro} and hope that this method can be extended to study similar problems. 

One of the main difficulties of such problems involving Dirac's operators is the strong indefiniteness together with a critical coupling with the ``bosonic'' non-linearity. 
Although the critical point theory for indefinite functionals was studied in  abstract form by Benci--Rabinowitz~\cite{bencirabinowitz1979critical},  Benci~\cite{benci1982critical}, for Schr\"odinger equations by Szulkin--Weth~\cite{szulkin2009ground} and also for Dirac's equations in planar form  by Bartsch--Ding~\cite{bartsch2006solutions}, Ding~\cite{ding2007variational} and on spin manifold with suitable nonlinearity by Isobe~\cite{isobe2010existence,isobe2011nonlinear,isobe2019onthemultiple}, most of the classical results do not apply directly here because of the nonlinear nature of the coupled problem. 
By this we mean that the Nehari manifolds employed here are not a linear space and the potential part is of exponential type.  

Part of the strategy in our previous work~\cite{jevnikar2020existence} about super Liouville equations can be used here. 
However, some  non-trivial new ideas are needed to handle problem~\eqref{eq:SShG-intro}. Special care is devoted to the study of the Palais--Smale condition, which is based on spectral decomposition and suitable test functions. Moreover, differently from the the super Liouville case in~\cite{jevnikar2020existence} we are able to treat the presence of harmonic spinors (i.e. ~$\dim \ker(\D)>0$) by exploiting the fact that the potential and its derivatives are bounded away from zero. This makes the argument technically more difficult both in the verification of the Palais--Smale condition and in the linking construction.
We can also solve~\eqref{eq:SShG-intro} on any closed Riemann surface, without genus restrictions as in~\cite{jevnikar2020existence}.
Finally, by exploiting the~$\bZ_2$-symmetry of the problem, we are able to produce a multiplicity result. In order to prove it, we need to 
construct an equivariant family of elements in the Nehari manifold with sufficiently low energy. To achieve such a property, 
it is in general convenient to consider scalar components of nearly constant absolute value, in an integral sense: to achieve a 
non-trivial family of this type we consider a {\em sweepout} of the surface via a thin interface, where the scalar component 
passes from the value $+1$ to the value $-1$. With these test functions at hand, we then employ 
 a min-max scheme of  {\em fountain-type}, see e.g.~\cite{willem1997minimax}, obtaining the following result. 

\begin{thm}\label{thm:main thm2}
Let~$M$ be a compact Riemann surface and suppose~$\rho\notin \Spect(\D)$. Then~\eqref{eq:SShG-intro} admits at least two geometrically distinct nonzero solutions.
\end{thm}

We postpone to Section~\ref{sec:mult} the discussion about how to distinguish variationally the solutions we produce.
As discussed in Remark \ref{r:last}, it is an interesting open problem to find further multiplicity results. 

\

The paper is organized as follows.
In Section~\ref{sec:preliminaries} we list some basic notions on the Dirac operator and the Sobolev spaces needed for our argument. 
Then we define the Nehari-type  constraint in Section~\ref{sec:setting} and check the Palais--Smale condition. 
After that, we can use a min-max principle to get the first (family of) nontrivial solutions in Section~\ref{sec:minimax solutions}, where the discussion splits according to the value of~$\rho$ and~$h=\dim \ker(\D_g)$. 
In the last section we use a fountain-type argument to obtain a second solution.

\

\noindent {\bf Acknowledgments.}
The third author would like to acknowledge Luciano Mari for helpful discussions on the CMC surfaces. 

A.M. has been partially supported by the project {\em Geometric problems with loss of compactness}  from Scuola Normale Superiore.
A.J. and A.M. have been partially supported by MIUR Bando PRIN 2015 2015KB9WPT$_{001}$.
They are also members of GNAMPA as part of INdAM.
R.W. is supported by the project DIP\_ECC\_MATE\_CoordAreaMate\_0495.

\section{Preliminaries}\label{sec:preliminaries}

We will assume some familiarity of spin geometry and classical Sobolev spaces, for which one can refer to e.g.~\cite{lawson1989spin, ginoux2009dirac, jost2011riemannian} and~\cite{gilbarg2001elliptic}. 
In our problem the spectral behaviors of the operators involved  plays a crucial role, and hence a discussion on this aspect will be included. 
Thus let us have a brief discussion on this.
  
Let~$M$ be a compact Riemann surface and~$g$ a metric in the given conformal class.
We know that there is always a spin structure over~$M$; let us fix one and denote it by~$P_{\Spin}(M,g)\to M$. 
The associated spinor bundle is a rank-four real vector bundle, denoted by~$\Sigma M$, whose construction depends on the choice of spin structure and the metric~$g$, both of which are fixed throughout our discussion. 
The important feature of~$\Sigma M$ is that it comes with a canonical Dirac bundle structure in the sense of~\cite[Definition 5.2]{lawson1989spin}: there exist canonical spinor metric~$\Abracket{\cdot,\cdot}$ (i.e. a fiberwise real inner product), a spin connection~$\snabla$ (induced from the Levi-Civita connection), and a Clifford multiplication~$\cliff: TM\to \End(\Sigma M)$ satisfying the Clifford relation
\begin{equation}\label{eq:Clifford relation}
 \cliff(X)\cliff(Y)+\cliff(Y)\cliff(X)=-2g(X,Y), \qquad \forall X,Y\in\Gamma(TM),
\end{equation}
and they are compatible with each other. 
Within these at hand, the Dirac operator~$\D=\D_g$ is the composition of the following operations (the middle isomorphism is given by the Riemannian metric~$g$) 
\begin{equation}
 \Gamma(S)\xrightarrow{\snabla} \Gamma(T^*M\otimes S)\xrightarrow{\cong} \Gamma(TM\otimes S)\xrightarrow{\cliff}\Gamma(S),
\end{equation}
and can be locally expressed in terms of a local orthonormal frame~$(e_i)$ as 
\begin{align}
 \D\psi = \sum_{i} \cliff(e_i)\snabla_{e_i}\psi, \qquad \forall \psi\in \Gamma(\Sigma M). 
\end{align}
It is an (essentially) self-adjoint, elliptic differential operator, due to the choice of the above Clifford relation~\eqref{eq:Clifford relation}.
$\D$ is the fundamental operator for the description of fermionic particles, just as the Laplacian operator is for bosonic ones. 

There is a well-defined global endomorphism~$\omega\coloneqq \cliff(e_1)\cliff(e_2) \in \End(\Sigma M)$ which anti-commutes with~$\D$,
in the sense that: 
\begin{align}\label{eq:volume action}
 \D(\omega\cdot \psi) = -\omega\cdot \D\psi, \qquad \forall\psi\in\Gamma(\Sigma M). 
\end{align}
Moreover,~$\Sigma M$ admits a quaternionic structure; 
in particular there exists a three-dimensional family~$\mathcal{J}$ of almost complex structrues, with each~$ \mathbf{j}\in \mathcal{J}$ commuting with~$\D$:
\begin{align}\label{eq:quaternionic action}
 \D(\mathbf{j}(\psi))= \mathbf{j}(\D\psi),\qquad \forall\psi\in\Gamma(\Sigma M). 
\end{align}
This gives an~$\sph^3$ (which is a Lie group) action on~$(\Sigma M, \D)$.  

Since the surface is compact, the spectrum of~$\D$, denoted by~$\Spect(\D)$, consists of eigenvalues.
Note that~$\D$ may have nontrivial kernel~$\ker(\D)$, whose elements are called \emph{harmonic spinors}. 
The dimension~$h=\dim\ker(\D)$ is necessarily finite and also conformally invariant: its value depends on the spin structure and conformal class. 

For later convenience, let~$\lambda_j$, $j\in \bZ_*=\bZ\setminus \{0 \}$ be the nonzero eigenvalues listed with multiplicities and in a non-decreasing order as
\begin{align}
  -\infty \leftarrow\cdots\le \lambda_{-k-1}\le \lambda_{-k}\le\cdots\le \lambda_{-1}\le 0 
 \le \lambda_1\le \cdots \le \lambda_k \le \lambda_{k+1}\le \cdots \to +\infty,
\end{align}
and let~$\lambda_0^l=0$ with~$1\le l\le h$ denote the zero eigenvalues (if any) counted with multiplicity. 
The corresponding eigenspinors are denoted by~$\Psi_j$,~$j\in \bZ_*$ and~$\Psi_{0,l}$,~$1\le l\le h$ if~$h\neq 0$, respectively. 
We may assume that these eigenspinors form a complete ~$L^2-$orthonormal basis for the~$L^2$ spinors.
The property~\eqref{eq:volume action} of~$\omega$ implies that~$\lambda_{-k}= -\lambda_k$ and~$\Psi_{-k}= \omega\cdot\Psi_k$, for any~$k\in \bZ_*$, while the existence of quaternionic structures and~\eqref{eq:quaternionic action} imply that any eigenvalue has multiplicity at least 3.

With respect to the above basis, any~$\psi\in L^2(\Sigma M)$ can be uniquely written as 
\begin{align}
 \psi= \sum_{j\in \bZ_*} a_j \Psi_j 
      +\sum_{0\le l\le h} b_l\Psi_{0,l}. 
\end{align}
If~$\psi \in C^1$ (or in~$H^1=W^{1,2}$), then
\begin{align}
 \D\psi =\sum_{j\in\bZ_*} a_j \lambda_j \Psi_j. 
\end{align}
This motivates the definition of fractional Dirac's operators as in~\cite{ammann2003habilitation}: for any~$s>0$, define~$|\D|^s\colon \Gamma(\Sigma M)\to \Gamma(\Sigma M)$ by 
\begin{align}
 |\D|^s \psi 
 = \sum_{j\in \bZ_*} |\lambda_j|^s a_j\Psi_j. 
\end{align}
The completion of~$\Gamma(\Sigma M)$ with respect to the inner product 
\begin{align}
 \Abracket{\psi,\phi}_{H^s} \coloneqq 
 \Abracket{\psi,\phi}_{L^2} + \Abracket{|\D|^s\psi, |\D|^s\phi}_{L^2}
\end{align}
is defined as the Sobolev space of spinors of order~$s$:
\begin{align}
 H^s(\Sigma M)\coloneqq 
 \braces{ \psi\in L^2(\Sigma M) \mid 
  \Abracket{\psi,\psi}_{H^s}<\infty}. 
\end{align}
Note that for~$s\in\mathbb{N}$, the space~$H^s(\Sigma M)$ coincides with the classical Sobolev space of spinor $W^{s,2}(\Sigma M)$ defined via covariant derivatives, so there is no confusion of notation. 
Furthermore, for~$-s<0$, the space~$H^{-s}(\Sigma M)$ is as usual defined to be the dual space~$(H^s(\Sigma M))^*$.
The Sobolev embedding theorem continues to hold in this setting.
In particular, for~$s\in (0,1)$, the space~$H^s(\Sigma M)$ continuously embeds into~$L^q(\Sigma M)$ for~$1\le q\le\frac{2}{1-s}$ and compactly embeds in~$L^q(\Sigma M)$ for~$1\le q<\frac{2}{1-s}$. 

The natural space of spinors to work with for Dirac's operators  is~$H^{1/2}(\Sigma M)$. 
According to the signs of the eigenvalues, we have the following decomposition
\begin{align}
 H^{\frac{1}{2}}(\Sigma M) = H^{\frac{1}{2},+}(\Sigma M)\oplus H^{\frac{1}{2},0}(\Sigma M) \oplus H^{\frac{1}{2},-}(\Sigma M),
 \qquad \psi= \psi^+ + \psi^0+\psi^-,
\end{align}
where~$H^{1/2,\pm}(\Sigma M)$ denotes the closure of the subspaces spanned by eigenspinors of positive resp. negative eigenvalues, while~$H^{1/2,0}(\Sigma M)$ is the~$h$-dimensional subpaces spanned by harmonic spinors (note that harmonic spinors are automatically smooth).
Moreover, given a positive~$\rho \notin\Spect(\D)$, we further split the space~$H^{1/2,+}(\Sigma M)$ into 
\begin{align}
 H^{\frac{1}{2},+}(\Sigma M)
 = H^{\frac{1}{2},+}_a(\Sigma M)\oplus H^{\frac{1}{2},+}_b(\Sigma M) 
\end{align}
with~$H^{1/2,\pm}_a(\Sigma M)$ being the closure (in~$H^{1/2}$) of the subspace spanned by eigenspinors corresponding to eigenvalues ``above'' respectively ``below''~$\rho$, namely:~$\lambda_j> \rho$ for~$H^{1/2,+}_a$ and~$0<\lambda_j< \rho$ for~$H^{1/2,+}_b$. 
Thus a spinor~$\psi\in H^{1/2}(\Sigma M)$ can be accordingly decomposed as 
\begin{align}\label{eq:decomp psi wrt rho}
 \psi= \psi^+_a +\psi^+_b + \psi^0 +\psi^-,
\end{align}
with a self-explaining notation. 

On the various subspaces of~$H^{1/2}(\Sigma M)$, the Dirac operators behaves differently, by definition. 
For example, we have 
\begin{align}
 \|\psi^+\|_{H^{1/2}} = \sum_{j>0} (1+\lambda_j) a_j^2, 
\end{align}
hence 
\begin{align}
 \int_M \Abracket{\D\psi^+,\psi^+}\dv_g 
  =\sum_{j>0} \lambda_j a_j^2 
  \ge \braces{\inf_{j>0}\frac{\lambda_j}{1+\lambda_j}} \|\psi^+\|_{H^{1/2}}^2
  = \frac{\lambda_1}{1+\lambda_1} \|\psi^+\|_{H^{1/2}}^2. 
\end{align}
We will analyze the concrete cases when we encounter them in the sequel.

For further materials on such spinors, and also the Moser-Trudinger embedding for~$H^1$ functions, see~\cite{ammann2003habilitation, aubin1998somenonlinear}, and also the preliminary part of~\cite{jevnikar2020existence}. 

Finally, we add a brief discussion on eigenspinors of constant length.
They give rise to the semi-trivial solutions of~\eqref{eq:SShG-intro}. 
That is, if~$(u,\psi)$ is a solution and~$u=const.$, then the first equation tells us that~$|\psi|=\rho\cosh(u)=const.$, while the second equation says~$\psi$ is an eigenspinor of the eigenvalues~$\rho\cosh(u)$. 
There are, indeed, such examples: on the round sphere~$\sph^2$, the Killing spinors are of constant length and are eigenspinors. 
They are the first observed solutions. 
Moreover, N. Kapouleas~\cite{kapouleas1990complete, kapouleas1991compact} showed that for a compact surface of genus~$\ge 3$, there are infinitely many immersions into~$\R^3$ with constant mean curvature~$H>0$. 
They give rise to eigenspinors of constant length on the immersed surface, and hence induce solutions of~\eqref{eq:SShG-intro} with~$u$ being a suitable constant. 
However, it is necessary that the Gauss curvature of the indueced metric for such immersions changes signs. 
Actually, for a compact surface with negative Gaussian curvature~$K(g)\le c<0$, this cannot happen, for the following reasons. 
Using the Weierstrass representation of CMC surfaces, T. Friedrich~\cite{friedrich1998spinorrepresentation} showed that: on a surface~$M$ and for a function~$H\in C^\infty(M)$,  there is a spinor~$\psi$ satisfying~$\D\psi= H\psi$ if and only if there is an immersion of~$\widetilde{M}$ (the universal cover of~$M$) into~$\R^3$ with mean curvature~$H$.  
However, the Efimov theorem says that there is no isometric ~$C^2$ immersed complete surfaces in~$\R^3$ with~$K(g)\le const. < 0$. 
This guarantees us a nontrivial solution on a surface of genus greater than one and with Gaussian curvature~$K(g)\le c<0$, once we find a nonzero solution.

\section{A variational setting}\label{sec:setting}
On the compact surface~$(M,g)$ with  spinor bundle~$\Sigma M$, we consider the action functional 
\begin{align}
 J_\rho \colon H^1(M)\times H^{\frac{1}{2}}(\Sigma M) \to \R
\end{align}
given by~\eqref{eq:action functional}:
\begin{align}\tag{\ref{eq:action functional}}
 J_\rho(u,\psi)
 = \int_M \left( |\nabla u|^2+8\Abracket{\D\psi,\psi}-8\rho\cosh(u)|\psi|^2 
 + 4\rho^2 \sinh(u)^2 \right) \dv_g. 
\end{align}
Its first variation is the following 
\begin{align}
 \dd J_\rho(u,\psi)[v,\phi]
 =&\int_M (-2\Delta u+ 4\rho^2\cdot 2\sinh(u)\cosh(u)-8\rho\sinh(u)|\psi|^2) v \dv_g \\
  &\qquad +\int_M 16\Abracket{\D\psi-\rho\cosh(u)\psi,\phi}\dv_g. 
\end{align}
Thus the Euler--Lagrange equations are 
\begin{equation*}\label{eq:SShG}\tag{\ref{eq:SShG-intro}}
 \begin{cases}
  \Delta u= 2\rho^2 \sinh(2u)- 4\rho \sinh(u)|\psi|^2, \\
  \D\psi= \rho \cosh(u)\psi. 
 \end{cases}
\end{equation*}

The difficulty in dealing with such equations is due to the strong indefiniteness of the Dirac operator, and a typical useful strategy is to use some Nehari type manifold to kill most of the negative directions, see e.g.~\cite{maalaoui2017characterization,jevnikar2020existence, jevnikar2020existence-sphere} and also~\cite{szulkin2009ground,szulkin2010themethod} for a more general treatment.
Here we will adopt the same approach. 
The outline of the proof here is a refinement of the argument introduced for super Liouville equations in~\cite{jevnikar2020existence}.

Define the set
\begin{align}
 N_\rho\coloneqq 
 \braces{ (u,\psi)\in H^1(M)\times H^{\frac{1}{2}}(\Sigma M) \mid P^- (1+|\D|)^{-1}(\D\psi-\rho\cosh(u)\psi)=0 }, 
\end{align}
which is clearly non-empty. 
Moreover, for each fixed~$u\in H^1(M)$, the above constraint  gives a vector space
\begin{equation}
 N_{\rho,u}\coloneqq \{\psi\in H^{\frac{1}{2}}(\Sigma M)\mid P^- (1+|\D|)^{-1}(\D\psi-\rho\cosh(u)\psi)=0  \}
\end{equation}
consisting of spinors lying in the kernel of the linear operator~$P^-(1+|\D|)^{-1}(\D-\rho\cosh(u))$. 
Therefore, we have a fibration
\begin{align}
 N_{\rho, u}\hookrightarrow N_\rho \to H^1(M).
\end{align}
This tells us that~$N_\rho$ has a vector bundle structure and is globally homeomorphic to a Hilbert space. 
We want next to understand some properties of the functional functional~$J_\rho$ restricted to~$N_\rho$.

An equivalent but useful description of~$N_\rho$ is given as follows. 
Define
\begin{align}
 G\colon  H^1(M)\times H^{\frac{1}{2}}(\Sigma M) 
 \to & H^{\frac{1}{2},-}(\Sigma M), \\
 (u,\psi) \mapsto & G(u,\psi)= P^-(1+|\D|)^{-1}(\D\psi-\rho\cosh(u)\psi).
\end{align}
Then~$N_\rho = G^{-1}(0)$ is a level set. 

\begin{lemma}
 $N_\rho$ is a smooth submanifold of~$H^1(M)\times H^{\frac{1}{2}}(\Sigma M)$. 
\end{lemma}
    \begin{proof}
     We show that for any~$(u,\psi)\in N_\rho$, the differential~$\dd G(u,\psi)\colon H^1(M)\times H^{\frac{1}{2}}(\Sigma M)\to H^{\frac{1}{2},-}(\Sigma M)$ is surjective. 
     Indeed, for any~$(v,\phi)\in H^1(M)\times H^{\frac{1}{2}}(\Sigma M)$, we have 
     \begin{align}
      \dd G(u,\psi)[v,\phi]
      = P^-(1+|\D|)^{-1}(\D\phi-\rho\cosh(u)\phi-\rho\sinh(u)v\psi). 
     \end{align}
     In particular, for any~$\phi\in H^{\frac{1}{2},-}(\Sigma M)$, the quadratic form $\Abracket{\dd G(u,\psi)[0,\phi], \phi}_{H^{\frac{1}{2}}}$ satisfies 
     \begin{align}
      \Abracket{\dd G(u,\psi)[0,\phi], \phi}_{H^{\frac{1}{2}}}
      =&\int_M \Abracket{\D\phi,\phi}\dv_g -\rho\int_M \cosh(u)|\phi|^2\dv_g  \\
      \le& -C_1 \|\phi\|_{H^{1/2}}^2 -\rho\int_M \cosh(u)|\phi|^2\dv_g 
     \end{align}
     and therefore it is non-degenerate. 
     It follows that~$\dd G(u,\psi)$ is surjective, and~$N_\rho= G^{-1}(0)$ is a smooth submanifold.  
    \end{proof}
    
Consider the constrained functional~$J_\rho|_{N_\rho}$, which has the advantage of  being no longer strongly indefinite. 
Let~$(u,\psi)\in N_\rho$ be a constrained critical point, namely 
\begin{align}
 0=\nabla^{N_\rho} J_\rho (u,\psi) 
 = \dd J_\rho(u,\psi)
   - 16\Abracket{\dd G(u,\psi),\varphi}
\end{align}
for some Lagrange multiplier~$\varphi=\varphi(u,\psi)\in H^{\frac{1}{2},-}(\Sigma M)$, where the coefficient~$16$ has been added for later convenience. 
This is to say, for any~$(v,\phi)\in H^1(M)\times H^{\frac{1}{2}}(\Sigma M)$, 
\begin{align}
 0= &\nabla^{N_\rho} J_\rho (u,\psi) [v,\phi]
 = \dd J_\rho(u,\psi)[v,\phi]
   - \Abracket{\dd G(u,\psi)[v,\phi],\varphi} \\
 =& \int_M \left[ (-2\Delta u+ 4\rho^2\cdot 2\sinh(u)\cosh(u)-8\rho\sinh(u)|\psi|^2) v 
 + 16\Abracket{\D\psi-\rho\cosh(u)\psi,\phi} \right] \dv_g \\
 &-16\int_M \left( -\rho\sinh(u)\Abracket{\psi,\varphi}+\Abracket{\D\varphi-\rho\cosh(u)\varphi,\phi} \right) \dv_g. 
\end{align}
Therefore,~$(u,\psi)$ satisfies 
\begin{align}\label{eq:EL:constrained:u}
 -\Delta u+2\rho^2\sinh(2u)-4\rho\cosh(u)|\psi|^2 
 + 8\rho\sinh(u)\Abracket{\psi,\varphi}=0, 
\end{align}
\begin{align}\label{eq:EL:constrained:psi}
 \D\psi-\rho\cosh(u)\psi
 -(\D\varphi-\rho\cosh(u)\varphi)=0, 
\end{align}
for some~$\varphi=\varphi(u,\psi)\in H^{\frac{1}{2},-}(\Sigma M)$. 
\begin{lemma}\label{lemma:naturality}
 If~$(u,\psi)\in H^1(M)\times H^{\frac{1}{2}}(\Sigma M)$ satisfies~\eqref{eq:EL:constrained:u}-\eqref{eq:EL:constrained:psi} for some~$\varphi\in H^{\frac{1}{2},-}(\Sigma M)$, then~$\varphi=0$ and hence~$(u,\psi)$ solves~\eqref{eq:SShG}. 
\end{lemma}

In other words, if~$(u,\psi)$ is a constrained critical point, then it is automatically a free critical point.

\begin{cor}
 $N_\rho$ is a Nehari-type manifold for~$J_\rho$. 
\end{cor}

 \begin{proof}[Proof of Lemma~\ref{lemma:naturality}]
  This follows from testing  equation~\eqref{eq:EL:constrained:psi} against~$\varphi$, which is admissible, so that 
  \begin{align}
   \int_M \Abracket{\D\psi-\rho\cosh(u)\psi,\varphi}\dv_g 
   -\int_M\Abracket{\D\varphi-\rho\cosh(u)\varphi,\varphi}\dv_g =0.
  \end{align}
  Since~$(u,\psi)\in N_\rho$, the first integral vanishes; meanwhile the second integral is equivalent to~$\|\varphi\|_{H^{1/2}}^2$: 
  \begin{align}
   0=\int_M\Abracket{\D\varphi-\rho\cosh(u)\varphi,-\varphi}\dv_g\le -C_1\|\varphi\|^2_{H^{1/2}} -\rho\int_M\cosh(u)|\phi|^2\dv_g\le 0. 
  \end{align}
  Thus we conclude that~$\varphi=0$. 
 \end{proof}

Another ingredient for the min-max procedure is the Palais--Smale condition. 
This is not valid for the special values~$\rho\in \Spect(\D)$: indeed, if~$\rho=\lambda_k\in\Spect(\D)$, then~$(0,t\Psi_k)$ form a $(PS)_0$ sequence while~$t\in \R$ can be arbitrarily large.
Actually it is a sequence of solutions which is unbounded as~$t\to\infty$.
Fortunately we can verify it as long as~$\rho\notin\Spect(\D)$.    
We are inspired here by some results in~\cite{bartsch2006solutions, maalaoui2017characterization, struwe2008variational}.  

\begin{prop}   
 For~$\rho\notin\Spect(\D)$ and~$\rho >0$, the functional~$J_\rho|_{N_\rho}$ satisfies the Palais--Smale condition. 
\end{prop}

    \begin{proof}
     Let~$c\in\R$ and let~$(u_n,\psi_n)$ be a~$(PS)_c$ sequence, namely verifying 
     \begin{align}\label{eq:PS:level}
      J_\rho(u_n,\psi_n)=
      \int_M |\nabla u_n|^2 + 4\rho^2\sinh(u_n)^2 + 8\Abracket{\D\psi_n-\rho\cosh(u_n)\psi_n,\psi_n} \dv_g \to c, 
    \end{align}
    \begin{align}\label{eq:PS:constraint}
     P^-(1+|\D|)^{-1}(\D\psi_n-\rho\cosh(u_n)\psi_n)=0, 
    \end{align}
    \begin{multline}\label{eq:PS:u}
    -2\Delta u_n + 4\rho^2\sinh(2u_n)-8\rho\sinh(u_n)|\psi_n|^2 +16\rho\sinh(u_n)\Abracket{\psi_n,\varphi_n}=\alpha_n \to 0
     \mbox{ in } H^{-1}(M), 
    \end{multline}
    \begin{align}\label{eq:PS:psi}
    \D\psi_n-\rho\cosh(u_n)\psi_n 
    - (\D\varphi_n-\rho\cosh(u_n)\varphi_n)
    =\beta_n \to 0 
    \quad \mbox{ in } \; H^{-\frac{1}{2}}(\Sigma M), 
    \end{align}
    where~$\varphi_n=\varphi_n(u_n,\psi_n) \in H^{\frac{1}{2},-}(\Sigma M)$ are the corresponding Lagrange multipliers.  
    We need to find a subsequence which converges to a solution. 
    
    Step 1: We first show that~$(u_n,\psi_n)$ is uniformly bounded in~$H^1(M)\times H^{\frac{1}{2}}(\Sigma M)$. 
    This is achieved by using a spectral decomposition and applying suitable test functions to the above equations~\eqref{eq:PS:level}-\eqref{eq:PS:psi}. 
    
    \begin{enumerate}
     \item[(i)] Testing~\eqref{eq:PS:psi} against~$\varphi_n \in H^{\frac{1}{2},-}(\Sigma M)$, and recalling~\eqref{eq:PS:constraint}, we get 
      \begin{align}
       \underbrace{-\int_M \Abracket{\D\varphi_n,\varphi_n}\dv_g}_{\ge C\|\varphi_n\|^2} 
       +\rho\int_M \cosh(u_n)|\varphi_n|^2 \dv_g =\Abracket{\beta_n,\varphi_n}
       =o(\|\varphi_n\|), 
      \end{align}
     which implies that 
    \begin{align}
     \|\varphi_n\| \to 0, & & \mbox{ and } & &
     \rho\int_M \cosh(u_n)|\varphi_n|^2\dv_g \to 0. 
    \end{align}
    
    \item[(ii)] Testing~\eqref{eq:PS:psi} against~$\psi_n$, again using~\eqref{eq:PS:constraint}, we get
    \begin{align}
     \int_M \left( \Abracket{\D\psi_n,\psi_n}-\rho\cosh(u_n)|\psi_n|^2 \right) \dv_g = \Abracket{\beta_n,\psi_n}= o(\|\psi_n\|). 
    \end{align}
    
    \item[(iii)] Substituting the above into~\eqref{eq:PS:level}, we see that 
    \begin{align}\label{eq:PS:level:psi}
     \int_M \left( |\nabla u_n|^2+4\rho^2\sinh(u_n)^2 \right) \dv_g
     +o(\|\psi_n\|) = c+ o(1).
    \end{align}
    Since~$\sinh(t)^2\ge t^2$ for any~$t\in \R$, it follows that 
    \begin{align}
     C(\rho)\|u\|_{H^1}^2 \le \int_M \left( |\nabla u_n|^2 + 4\rho^2 u^2 \right) \dv_g \le c+ o(1) + o(\|\psi_n\|). 
    \end{align}
    
     \item[(iv)]  Testing~\eqref{eq:PS:u} against~$\tanh(u/2)=\frac{\sinh(u/2)}{\cosh(u/2)} \in [-1,1]$, using~$\tanh(t)'=\cosh(t)^{-2}\le 1$, we get
    \begin{align}
     \int_M \left( \frac{|\nabla u_n|^2}{\cosh(u_n/2)^2} \right. 
     & \left. +16\rho^2\cosh(u_n)\sinh(\frac{u_n}{2})^2 
      -16\rho\sinh(\frac{u_n}{2})^2 |\psi_n|^2 
      -32\rho\sinh(\frac{u_n}{2})^2\Abracket{\psi_n,\varphi_n} \right) \dv_g 
      \\
      = & \Abracket{\alpha_n, \tanh(\frac{u_n}{2})}_{H^{-1}\times H^1} = o(1). 
    \end{align}
    Since 
    \begin{align}
     4\cosh(u_n)\sinh(\frac{u_n}{2})^2
     =&2\cosh(u_n)(\cosh(u_n)-1)
       =\cosh(2u_n)+1-2\cosh(u_n) \\
     =&2\sinh(u_n)^2-2\sinh(\frac{u_n}{2})^2 
       \le 2\sinh(u_n)^2, 
    \end{align}
    we get
    \begin{align}
     0\le 16\rho^2\cosh(u_n)\sinh(\frac{u_n}{2})^2 \le 8\rho^2\sinh(u_n)^2. 
    \end{align}
    Using 
    \begin{align}
     \int_M 32\rho\sinh(\frac{u_n}{2})^2\Abracket{\psi_n,\varphi_n} \dv_g 
     \le & 16\rho\int_M \sinh(\frac{u_n}{2})^2|\psi_n|^2\dv_g \\  
       &+ 16\rho\int_M \sinh(\frac{u_n}{2})^2|\varphi_n|^2\dv_g 
    \end{align}
    with
    \begin{align}
     16\rho\int_M \sinh(\frac{u_n}{2})^2|\varphi_n|^2\dv_g 
     \le 8\rho \int_M \cosh(u_n)|\varphi_n|^2\dv_g =o(1),
    \end{align} 
    we obtain 
    \begin{align}
     \int_M 32\rho\sinh(\frac{u_n}{2})^2 |\psi_n|^2 
     \le  &\int_M  \left( \frac{|\nabla u_n|^2}{\cosh(u_n/2)^2} 
      +16\rho^2\cosh(u_n)\sinh(\frac{u_n}{2})^2 \right) 
      \dv_g  +o(1) \\
      \le & \int_M \left( |\nabla u_n|^2 + 8\rho^2\sinh(u_n)^2 \right) \dv_g 
           +o(1) \\
      \le & 2c+ o(1) + o(\|\psi_n\|).
    \end{align}
     
    \item[(v)] Up to now it remains to estimate~$(\|\psi_n\|)_{n\ge 1}$.
    We will split this into two cases:~$0<\rho<\lambda_1$ and~$\lambda_k<\rho<\lambda_{k+1}$ for some~$k\ge 1$. 
    
    \
    
    \begin{enumerate}
     \item[(v-1)] Consider the case~$\rho\in (0, \lambda_1)$ and write each~$\psi_n$ as 
     \begin{align}
      \psi_n= \psi_n^+ +\psi_n^0 + \psi_n^-.  
     \end{align}
      Testing~\eqref{eq:PS:psi} against~$\psi_n^+$ we obtain: 
    \begin{align}  
     \int_M & \Abracket{\D\psi-\rho\psi,\psi_n^+} \dv_g \\
     = &\int_M \left( \rho(\cosh(u_n)-1)\Abracket{\psi_n,\psi_n^+}
       -\Abracket{\D\varphi_n-\rho\cosh(u_n)\varphi_n,\psi_n^+} \right) \dv_g 
       +\Abracket{\beta_n,\psi_n^+}_{H^{-\frac{1}{2}}\times H^{\frac{1}{2}}}\\
     =&\int_M 2\rho\sinh(\frac{u_n}{2})^2\Abracket{\psi_n,\psi_n^+} \dv_g 
             + o(\|\psi_n^+\|)  \\
     \le& \parenthesis{\int_M \frac{\rho^2}{16}\sinh(\frac{u_n}{2})^4 \dv_g}^{\frac{1}{4}} 
     \parenthesis{\int_M 16\rho\sinh(\frac{u_n}{2})^2 |\psi_n|^2v \dv_g}^{\frac{1}{2}}
     \parenthesis{\int_M |\psi_n^+|^4 \dv_g}^{\frac{1}{4}}
     + o(\|\psi_n\|). 
    \end{align}
    Since
    \begin{align}
     4\sinh(\frac{u_n}{2})^4
     =\parenthesis{2\sinh(\frac{u_n}{2})^2 }^2 
     =\sinh(u_n)^2 - 4\sinh(\frac{u_n}{2})^2 
     \le \sinh(u_n)^2, 
    \end{align}
    we get  
    \begin{align}
     \int_M & \Abracket{\D\psi-\rho\psi,\psi_n^+} \dv_g
     \le C\parenthesis{c+ o(1) + o(\|\psi_n\|) }^{\frac{1}{4}+\frac{1}{2}} \|\psi_n^+\|_{H^{1/2}} 
      + o(\|\psi_n\|).
    \end{align}
    Note that, 
    \begin{align}
     \int_M \Abracket{\D\psi_n,\psi_n^+} \dv_g
     \ge \frac{\lambda_1}{\lambda_1+1} \|\psi_n^+\|^2_{H^{1/2}},
    \end{align}
    hence  if $ \rho<\lambda_1$,
    \begin{align}
     \int_M \Abracket{\D\psi-\rho\psi_n,\psi_n^+} \dv_g
     \ge \parenthesis{1-\frac{\rho}{\lambda_1}}\frac{\lambda_1}{\lambda_1+1} \|\psi_n^+\|^2_{H^{1/2}}. 
    \end{align}
    Therefore, in the case~$\rho<\lambda_1$, it follows that 
    \begin{align}
     \parenthesis{1-\frac{\rho}{\lambda_1}}\frac{\lambda_1}{\lambda_1+1} \|\psi_n^+\|^2_{H^{1/2}}
     \le C\parenthesis{c+o(1)+o(\|\psi_n\|)}^{\frac{3}{4}} \|\psi_n^+\|_{H^{1/2}}
     + o(\|\psi_n\|). 
    \end{align}
    
     The negative parts~$\psi_n^-$ can be similarly estimated using~\eqref{eq:PS:constraint}
    \begin{align}
     -\int_M \Abracket{\D\psi_n-\rho\psi_n,\psi_n^-}
     = & -\rho\int_M (\cosh(u_n)-1)\Abracket{\psi_n,\psi_n^-}\dv_g \\
     \le & C\parenthesis{c+o(1)+o(\|\psi_n\|)}^{\frac{3}{4}} \|\psi_n^-\|_{H^{1/2}}. 
    \end{align}
    Recall that~$\lambda_{-1} = -\lambda_1$, and hence 
    \begin{align}
     -\int_M \Abracket{\D\psi_n-\rho\psi_n,\psi_n^-}
     \ge  \parenthesis{1+\frac{\rho}{\lambda_1}}\frac{\lambda_1}{\lambda_1+1} \|\psi_n^-\|^2_{H^{1/2}}
     \ge \frac{\lambda_1}{\lambda_1+1} \|\psi_n^-\|^2_{H^{1/2}}.
    \end{align}
    So,  we get an estimate for~$\psi_n^-$ without tail terms which is also independent of~$\rho$:
    \begin{align}
     \frac{\lambda_1}{\lambda_1+1} \|\psi_n^-\|^2_{H^{1/2}}
     \le C\parenthesis{c+o(1)+o(\|\psi_n\|)}^{\frac{3}{4}} \|\psi_n^-\|_{H^{1/2}}. 
     \end{align}
     As for the harmonic parts~$\psi^0_n$, since they are orthogonal to~$\psi_n^+ + \psi_n^-$ with respect to the~$L^2$ global inner product,  
     we have~$\|\psi_n^0\|_{L^2} \le \|\psi_n^0\|_{L^2}$: we can then use~\eqref{eq:PS:psi} to get
     \begin{align}
      \rho\int_M |\psi_n^0|^2\dv_g 
      \le &\int_M \rho|\psi_n|^2\dv_g  
           \le \int_M \rho\cosh(u_n)|\psi_n|^2 \dv_g \\
        =&\int_M\Abracket{\D\psi_n,\psi_n}\dv_g + o(\|\psi_n\|) \\
        =& \int_M \Abracket{\D\psi_n^+ ,\psi_n^+} \dv_g  
        + \underbrace{\int_M \Abracket{\D\psi_n^-,\psi_n^-}\dv_g }_{\le 0} 
        + o(\|\psi_n\|) \\ 
       \le& \|\psi_n^+\|^2_{H^{1/2}}  + o(\|\psi_n\|). 
     \end{align}
     As the space of harmonic spinors has finite dimension, any two norms on it are equivalent; in particular
     \begin{align}
      \rho\|\psi_n^0\|_{H^{1/2}}^2 
      \le  C\|\psi_n^0\|^2_{L^2} 
      \le C\|\psi_n^+\|^2_{H^{1/2}} + o(\|\psi_n\|). 
     \end{align}     
     Since~$\|\psi_n\|_{H^{1/2}}^2 = \|\psi_n^+\|_{H^{1/2}}^2 + \|\psi_n^-\|_{H^{1/2}}^2 + \|\psi^0\|^2_{H^{1/2}}$, we can add up the above estimates to get
     \begin{align}
      \qquad\qquad  \parenthesis{1-\frac{\rho}{\lambda_1}}\frac{\lambda_1}{\lambda_1+1} \|\psi_n\|^2_{H^{1/2}}
     \le C (1+\frac{1}{\rho})\parenthesis{c+o(1)+o(\|\psi_n\|)}^{\frac{3}{4}} \|\psi_n\|_{H^{1/2}}, 
     + o(\|\psi_n\|)
     \end{align}
     from which it follows that the~$\psi_n$'s are uniformly bounded:
     \begin{align}
      \|\psi_n\|_{H^{1/2}}\le C(c, \rho)<\infty. 
     \end{align}
     As a consequence, the norms~$\|u_n\|_{H^1}$ are also uniformly bounded. 
     
     Therefore, in the case~$0<\rho<\lambda_1$, we see that any~$(PS)_c$ sequence (for some~$c\in \R$) is bounded.
    
    \
    
     \item[(v-2)] Next we deal with the case of~$\rho$  large. Let~$\rho\in (\lambda_k,\lambda_{k+1})$ for some~$k\ge 1$. 
     Accordingly, we decompose the spinors~$\psi_n\in H^{\frac{1}{2}}(\Sigma M)$ as in~\eqref{eq:decomp psi wrt rho}:
     \begin{align}
      \psi_n= \psi^+_{an} + \psi^+_{bn}+\psi^0_n + \psi^-_n.
     \end{align}
     
     Testing~\eqref{eq:PS:psi} against~$\psi^+_{bn}=\sum_{0<j\le k} a_{n,j}\Psi_j$ we obtain:
     \begin{align}
      \qquad \qquad \int_M \Abracket{\D\psi_n,\psi^+_{bn}}\dv_g  
      = &\int_M \rho\cosh(u_n)\Abracket{\psi_n,\psi^+_{bn}}\dv_g \\ 
      & -\int_M \Abracket{\D\varphi_n-\rho\cosh(u_n)\varphi_n,\psi^+_{bn}}\dv_g +\Abracket{\beta_n,\psi^+_{bn}}_{H^{-\frac{1}{2}}\times H^{\frac{1}{2}}}\\
      =&\int_M \rho\cosh(u_n)\Abracket{\psi_n,\psi^+_{bn}}\dv_g + o(\|\psi^+_{bn}\|). 
     \end{align}
     It follows that 
     \begin{align}
      \int_M \rho\Abracket{\psi_n,\psi^+_{bn}}\dv_g 
      -&\int_M \Abracket{\D\psi_n,\psi^+_{bn}}\dv_g \\
      =&\int_M \rho(1-\cosh(u_n))\Abracket{\psi_n,\psi^+_{bn}}\dv_g + o(\|\psi^+_{bn}\|).
     \end{align}
     On one side, 
     \begin{align}
      \int_M \rho & \Abracket{\psi_n,\psi^+_{bn}}\dv_g 
      -\int_M \Abracket{\D\psi_n,\psi^+_{bn}}\dv_g 
         =\sum_{0<j\le k}(\rho-\lambda_j)a_{n,j}^2 \\
      =& \sum_{0<j\le k} \frac{\rho-\lambda_j}{1+\lambda_j} (1+\lambda_j)a_{n,j}^2 
        \ge \frac{\rho-\lambda_k}{1+\lambda_k}\sum_{0<j\le k}(1+\lambda_j) a_{n,j}^2  
        =\frac{\rho-\lambda_k}{1+\lambda_k}\|\psi^+_{bn}\|^2_{H^{1/2}}, 
     \end{align}
while on the other 
     \begin{align}
      \left|\int_M \rho(1-\cosh(u_n))\Abracket{\psi_n,\psi^+_{bn}}\dv_g\right|
      =&\int_M 2\rho\sinh(\frac{u_n}{2})^2|\psi_n| |\psi^+_{bn}| \dv_g \\
      \le & C(c+o(1)+o(\|\psi_n\|))^{\frac{3}{4}}\|\psi^+_{bn}\|_{H^{1/2}}. 
     \end{align}
     Thus we get  
     \begin{align}
      \frac{\rho-\lambda_k}{1+\lambda_k}\|\psi^+_{bn}\|^2_{H^{1/2}}
      \le C(c+o(1)+o(\|\psi_n\|))^{\frac{3}{4}}\|\psi^+_{bn}\|_{H^{1/2}}. 
     \end{align}
     
     Similarly we test~\eqref{eq:PS:psi} against~$\psi^+_{an}$ and subtract~$\int_M \rho\Abracket{\psi_n,\psi^+_{an}}\dv_g$ to get
     \begin{align}
      \int_M\Abracket{\D\psi_n-\rho\psi_n,\psi^+_{an}}\dv_g 
      = \int_M \rho(\cosh(u_n)-1) \Abracket{\psi_n,\psi^+_{an}}\dv_g + o(\|\psi^+_{an}\|).
     \end{align}
     Now, the left-hand side can be bounded from below by 
     \begin{align}
      \int_M\Abracket{\D\psi_n-\rho\psi_n,\psi^+_{an}}\dv_g 
      =\sum_{j>k} (\lambda_j-\rho)a_{n,j}^2
      \ge \frac{\lambda_{k+1}-\rho}{\lambda_{k+1}+1}\|\psi^+_{an}\|_{H^{1/2}}, 
     \end{align}
     while the right-hand side is estimated as before. 
     Hence 
     \begin{align}
       \frac{\lambda_{k+1}-\rho}{\lambda_{k+1}+1}\|\psi^+_{an}\|^2_{H^{1/2}}
       \le C(c+o(1)+o(\|\psi_n\|))^{\frac{3}{4}}\|\psi^+_{an}\|_{H^{1/2}}
       +o(\|\psi^+_{an}\|).
     \end{align}
     With the estimates for~$\psi_n^-$ and~$\psi_n^0$:
     \begin{align}
      \frac{\lambda_1}{\lambda_1+1} \|\psi_n^-\|^2_{H^{1/2}}
     \le C\parenthesis{c+o(1)+o(\|\psi_n\|)}^{\frac{3}{4}} \|\psi_n^-\|_{H^{1/2}}; 
     \end{align}
     \begin{align}
      \rho\|\psi_n^0\|_{H^{1/2}}^2
      \le C\|\psi_n^+\|^2 + o(\|\psi_n\|),
     \end{align}
     obtained in the same way as before, we again come to
     \begin{align}
      C(k, \rho)\|\psi_n\|^2_{H^{1/2}}
       \le C(c+o(1)+o(\|\psi_n\|))^{\frac{3}{4}}\|\psi_n\|_{H^{1/2}}
       +o(\|\psi_n\|), 
     \end{align}
     where
     \begin{align}
      C(k,\rho)= \min\braces{\frac{\rho-\lambda_k}{1+\lambda_k},  \frac{\lambda_{k+1}-\rho}{\lambda_{k+1}+1}, \frac{\lambda_1}{\lambda_1+1},\rho}. 
     \end{align}
     From this it follows that 
     \begin{align}
      \|\psi_n\|_{H^{1/2}} \le C(c,\rho)<\infty.
     \end{align}    
     
    \end{enumerate}    
    
    \end{enumerate}
    Therefore, the sequence~$(u_n,\psi_n)$ is shown to be uniformly bounded in~$H^1(M)\times H^{\frac{1}{2}}(\Sigma M)$. 
    
    Step 2: We extract a subsequence which converges to a smooth solution of~\eqref{eq:SShG}. 
    
    By Banach--Alaoglu's theorem there is a subsequence, still denoted as~$(u_n,\psi_n)$ for simplicity of notation, which converges weakly to, say,~$(u_\infty,\psi_\infty)\in H^1(M)\times H^{\frac{1}{2}}(\Sigma M)$. 
    By the compact embedding theorems due to Moser--Trudinger and to Rellich--Kondrachov, we have the following strong convergence:
    \begin{align}
     e^{au_n}\to e^{au_\infty}, \qquad &\mbox{ in } L^p(M),\qquad \forall p\in [1,\infty),\quad \forall a\in\R; \\
     \psi_n\to\psi_{\infty}, \qquad &\mbox{ in } L^q(\Sigma M), \qquad \forall q \in [1,4). 
    \end{align}
    Consequently,~$\cosh(u_n)\psi_n\to \cosh(u_\infty)\psi_\infty$ in~$L^p(\Sigma M)$ for any~$p<4$ and~$\sinh(u_n)|\psi_n|^2 \to \sinh(u_\infty)|\psi_\infty|^2$ in~$L^q(M)$ for any~$q<2$. 
    These conditions are strong enough to guarantee that~$(u_\infty,\psi_\infty)$ is a weak solution of~\eqref{eq:SShG}. 
    The standard regularity theory then applies to show that it is a smooth solution. 
    In particular,~$(u_\infty,\psi_\infty)\in N_\rho$. 
    
    It remains to show that the differences 
    \begin{align}
     v_n\coloneqq u_n -\ u_\infty, & & 
     \phi_n\coloneqq \psi_n-\psi_\infty
    \end{align}
    converge to~$(0,0)$ strongly in~$H^1(M)\times H^{1/2}(\Sigma M)$. 
    Indeed, 
    \begin{align}
     \Delta v_n=\Delta u_n -\Delta u_\infty 
     =&2\rho^2\parenthesis{\sinh(2u_n)-\sinh(2u_\infty)} 
      -4\rho\parenthesis{\cosh(u_n)|\psi_n|^2-\cosh(u_\infty)|\psi_\infty|^2} \\
      & +8\rho\sinh(u_n)\Abracket{\psi_n,\varphi_n}
      -\alpha_n, 
    \end{align}
    which converges to 0 in~$H^{-1}(M)$. 
    Since~$\|v_n\|_{L^2}\to 0$ as~$n\to \infty$, we conclude that~$v_n\to 0$ strongly in~$H^1(M)$. 
    The argument for the strong convergence for~$\phi_n$ to~$0$ in~$H^{1/2}(\Sigma M)$ is similar.
    
    Finally, since the topology of~$N_\rho$ is induced from that of~$H^1(M)\times H^{1/2}(\Sigma M)$, the (sub)sequence $(u_n,\psi_n)$ also converges to~$(u_\infty,\psi_\infty)$ in~$N_\rho$, verifying the Palais--Smale  condition. 
   \end{proof}

\section{Proof of Theorem~\ref{thm:main thm}: Min-max solutions}\label{sec:minimax solutions}

This section is devoted to the proof of the main theorem. 
We study the local geometry near the trivial solution~$(0,0)\in N_\rho$ and obtain some nontrivial ones. 
We will see that if there are no harmonic spinors, for~$\rho$ small the functional~$J_\rho|_{N_\rho}$ displays some mountain pass structure while if either harmonic spinors are present or~$\rho$ is large,~$J_\rho|_{N_\rho}$ shows a local linking structure. 
This is similar to the phenomena shown for super Liouville equations, see e.g.~\cite{jevnikar2020existence}, where the easier case with no harmonic spinors is considered. 

Note that the energy value of~$(0,0)$ is 
\begin{align}
 J_\rho(0,0)= 0.
\end{align}

\subsection{Local estimates}
Let~$(u,\psi)\in N_\rho\cap B_R(0,0)$ for some~$R>0$, where the distance is measured with respect to the Hilbert norm on~$H^1(M)\times H^{1/2}(\Sigma M)$. 
The constraint condition implies 
\begin{align}
 \int_M \Abracket{\D\psi-\rho\cosh(u)\psi, \psi^-}\dv_g =0, 
\end{align}
which helps to control the negative part~$\psi^-$: 
\begin{align}
 \|\psi^-\|_{H^{1/2}}^2
 \le&  -\int_M\Abracket{\D\psi^-,\psi^-}+\rho\int_M \cosh(u)|\psi^-|^2\dv_g 
 = -\rho\int_M \cosh(u)\Abracket{\psi^++\psi^0,\psi^-}\dv_g  \\
 \le& \rho \|\cosh(u)\|_{L^2}\|\psi^++\psi^0\|_{L^4}\|\psi^-\|_{L^4}
\end{align}
hence 
\begin{align}\label{eq:control of psi-}
 \|\psi^-\|_{H^{1/2}} \le \rho \|\cosh(u)\|_{L^2}\|\psi^++\psi^0\|_{H^{1/2}}. 
\end{align}
We next claim that the following estimate holds 
\begin{align}\label{eq:estimate for e(2u)}
 \|e^{u} - 1 \|_{L^2} 
 \le \| e^{|u|}\|_{L^4} \| u\|_{L^4} \le C(R) \|u\|_{H^1}.   
\end{align}
 Indeed, for each~$x\in M$, there exists~$\theta(x)\in [0,1]$ such that 
 \begin{align}
  |e^{u(x)}- e^0|= e^{\theta(x)u(x)} |u(x)|
  \le e^{|u(x)|} |u(x)|. 
 \end{align}
 By Moser-Trudinger's inequality, 
 \begin{align}
  \|e^{|u|}\|_{L^4}\le C\exp( C\|u\|_{H^1}^2) \le C(R),
 \end{align}
 as long as~$\|u\|_{H^1}\le R$. 
 Then \eqref{eq:estimate for e(2u)} follows.

As a consequence, possibly relabelling $C(R)$
\begin{align}
 \|\cosh(u)-1\|_{L^2} \le C(R)\|u\|_{H^1}. 
\end{align}

Consider now the functional~$J_\rho$, which can be decomposed into three parts:
\begin{align}\label{eq:decomp of J}
 J_\rho(u,\psi)
 =&\int_M \left( |\nabla u|^2+ 4\rho^2 \sinh(u)^2
        +8\Abracket{\D\psi,\psi}-8\rho\cosh(u)|\psi|^2 \right) \dv_g \\
 =&\int_M \left( |\nabla u|^2+ 4\rho^2 \sinh(u)^2
        +8\Abracket{\D\psi-\rho\cosh(u)\psi,\psi^+ +\psi^0} \right) 
        \dv_g \\
 =&\int_M \left( |\nabla u|^2 + 4\rho^2 \sinh(u)^2 \right) \dv_g 
    +\int_M  8\Abracket{\D\psi-\rho\psi, \psi^+ + \psi^0}\dv_g \\
  &\qquad   +8\rho\int_M \Abracket{(1-\cosh(u))\psi,\psi^+ + \psi^0}\dv_g, 
\end{align}
where we have used the constraint condition. The second part can be rewritten as
\begin{align}
\int_M  8\Abracket{\D\psi-\rho\psi, \psi^++\psi^0}\dv_g  
 =\int_M  8\Abracket{\D\psi^+-\rho\psi^+, \psi^+}\dv_g 
   -\int_M 8|\psi^0|^2\dv_g. 
\end{align}
Since~$\sinh(t)^2\ge t^2$, 
\begin{align}
 \int_M \left( |\nabla u|^2 + 4\rho^2 \sinh(u)^2 \right) \dv_g 
 \ge \int_M \left( |\nabla u|^2 + 4\rho^2 |u|^2 \right) \dv_g 
 \ge C(\rho) \|u\|_{H^1}^2. 
\end{align}
As observed, the third integral can be estimated by
\begin{align}
 \left|8\rho\int_M \Abracket{(1-\cosh(u))\psi,\psi}\dv_g\right|
 \le 8\rho\|\cosh(u)-1\|_{L^2} \|\psi\|_{L^4} 
 \le C(R)\|u\|_{H^1} \|\psi\|_{H^{1/2}}^2. 
\end{align}
This is cubic for~$(u,\psi)$ small in the Hilbert space~$H^1(M)\times H^{1/2}(\Sigma M)$, hence the functional~$J_\rho(u,\psi)$ is dominated by the other two terms in~\eqref{eq:decomp of J}.  
Note that we may relabel~$C(R)$ once more. 

To understand the local behavior of the second part in~\eqref{eq:decomp of J}, it is convenient to write~$\psi^+$ and~$\psi^0$ as a combination of eigenspinors
\begin{align}
 \psi^+ =\sum_{j=1}^\infty a_j\Psi_j, & & 
 \psi^0 =\sum_{l=1}^h b_l \Psi_{0,l}.
\end{align}
Then we have 
\begin{align}\label{eq:decomp:second part}
 \int_M 8\Abracket{\D\psi- \rho\psi,\psi}\dv_g 
 =& \sum_{j>0}8(\lambda_j-\rho)a_j^2  
   -\sum_{l=1}^h 8\rho b_l^2. 
\end{align}

\subsection{Mountain pass solutions}\label{sec:mountain pass}
First, we consider the case~$h=0$ (i.e. when there are no harmonic spinors) and~$0<\rho<\lambda_1$. 
We will see that~$J_\rho$ has local mountain pass geometry and thus admits mountain-pass solutions. 

In this case~$\psi=\psi^++\psi^-$, and 
\begin{align}
 \|\psi\|_{H^{1/2}}^2 
 =\|\psi^+\|^2_{H^{1/2}}+\|\psi^-\|^2_{H^{1/2}}
 \le (1+C(R))\|\psi^+\|_{H^{1/2}}^2. 
\end{align}
The second part in~\eqref{eq:decomp of J} becomes 
\begin{align}
 \int_M 8\Abracket{\D\psi- \rho\psi,\psi}\dv_g 
 =&\int_M 8\Abracket{\D\psi^+ -\rho\psi^+,\psi^+}\dv_g  \\
 =& \sum_{j>0}(\lambda_j-\rho) a_j^2 
   =\sum_{j>0}\frac{\lambda_j-\rho}{\lambda_j+1} (\lambda_j+1)a_j^2 \\
 \ge&\frac{\lambda_1-\rho}{\lambda_1+1} \|\psi^+\|_{H^{1/2}}\\
 \ge&\frac{1}{1+C(R)}\frac{\lambda_1-\rho}{\lambda_1 +1}\|\psi\|^2_{H^{1/2}}. 
\end{align}
Therefore, for~$r\equiv \|u\|_{H^1}+\|\psi\|_{H^{1/2}}$ so small that 
\begin{align}
 0<r < \frac{1}{2C(R)}\frac{1}{1+C(R)}\frac{\lambda_1-\rho}{\lambda_1 +1}
\end{align}
we can bound the functional~$J_\rho|_{N_\rho}$ from below by
\begin{align}
 J_\rho(u,\psi)
 \ge & C(\rho)\|u\|^2_{H^1} 
      +\frac{1}{1+C(R)}\frac{\lambda_1-\rho}{\lambda_1 +1}\|\psi\|^2_{H^{1/2}} 
      -C(R)\|u\|_{H^1}\|\psi\|_{H^{1/2}}^2  \\
 \ge & C(\rho)\|u\|^2_{H^1} 
      +\frac{1}{2}\frac{1}{1+C(R)}\frac{\lambda_1-\rho}{\lambda_1 +1}\|\psi\|^2_{H^{1/2}} \\
 \ge &\min\braces{C(\rho),\frac{1}{2}\frac{1}{1+C(R)}\frac{\lambda_1-\rho}{\lambda_1 +1} } \parenthesis{\|u\|^2_{H^1}+\|\psi\|_{H^{1/2}}^2 }.
\end{align}
That is,~$J_\rho|_{N_\rho}$ has a strict local minimum at~$(0,0)$ in~$N_\rho$ and for some small~$r_0$,  in the small neighborhood~$B_{r_0}(0,0)\cap N_\rho$,  one has 
\begin{align}
 J_\rho(u,\psi)\ge C(\rho, \lambda_1)\parenthesis{\|u\|^2_{H^1}+\|\psi\|_{H^{1/2}}^2 }.
\end{align}
On the other hand, we can find a negative level as follows. 
Take a constant function~$u=\bar{u}$ such that 
\begin{align}
 \rho\cosh(\bar{u})>\lambda_1+1,
\end{align}
and~$\psi=s \Psi_1$ with~$s\gg 1$ so that 
\begin{align}
 J_\rho(\bar{u},s \Psi_1)
 =& \int_M \left( 4\rho^2\sinh(\bar{u})^2 + 8(\lambda_1-\rho\cosh(\bar{u})) s^2 |\Psi_1|^2 \right) \dv_g \\
 =&4\rho^2\sinh(\bar{u})^2\Vol(M,g) -8(\rho\cosh(\bar{u})-\lambda_1)s^2<0. 
\end{align}
Since~$N_\rho$ is path-connected, we can find a~$C^1$ path inside~$N_\rho$ connecting~$(0,0)$ to~$(\bar{u},s\Psi_1)$.  
Let~$\Gamma$ be the space of such curves parametrized by the unit interval~$[0,1]$. 
Then the number 
\begin{align}
 c_1\coloneqq \inf_{\alpha\in\Gamma} \sup_{t\in [0,1]} J_\rho(\alpha(t)) >0
\end{align}
is a critical level, which means that we can find a critical point different from~$(0,0)$. 
This is the mountain pass solution  we are looking for.

As explained in the introduction, this is for sure a nontrivial solution on surfaces with negative curvature. 
But in general this might be a semi-trivial solution with~$u=$ constant.    
On the round sphere for example, the solutions give by~$(u=\mbox{const},\psi=\mbox{Killing spinor})$ actually corresponds to the mountain pass solutions here.

\subsection{Linking solutions}\label{sec:linking}
In this subsection we consider either the case~$h>0$ or~$\rho>\lambda_1$. 
As~\eqref{eq:decomp:second part} indicates, near~$(0,0)\in N_\rho$ there are some directions along which the functional~$J_\rho$ decreases. 
We will display a local linking geometry of~$J_\rho|_{N_\rho}$ and thus obtain min-max solutions of linking type. 
Without loss of generality, assume~$\lambda_k<\rho <\lambda_{k+1}$ for some~$k\in\mathbb{N}$, where~$k\ge 1$ if~$h=0$ and~$k$ could also be~$0$ if~$h>0$. 

Recall that we have introduced the decomposition
\begin{align}
 H^{\frac{1}{2}}(\Sigma M)
 = H^{\frac{1}{2},+}_a(\Sigma M)
  \oplus H^{\frac{1}{2},+}_b(\Sigma M)
  \oplus H^{\frac{1}{2},0}(\Sigma M)
  \oplus H^{\frac{1}{2},-}(\Sigma M),
\end{align}
and a  spinor~$\psi\in H^{1/2}(\Sigma M)$ is decomposed accordingly as 
\begin{align}\label{eq:decomposition of spinor}
 \psi =\psi^+_a + \psi^+_b  +\psi^0 + \psi^- .  
\end{align}
Consider the set
\begin{align}
 \sN_\rho 
 \coloneqq \{0\}\times\parenthesis{H^{\frac{1}{2},0} \oplus H^{\frac{1}{2},+}_b} (\Sigma M) 
 \subset H^1(M)\times H^{\frac{1}{2}}(\Sigma M).
\end{align}
This is actually a linear subspace contained inside~$N_\rho$, and 
\begin{align}
 J_\rho(0,\psi)
 =\int_M 8\rho\Abracket{\D\psi-\rho\psi,\psi} < 0 ,\qquad 
 \forall (0,\psi)\in \sN_\rho, \quad \psi\neq 0.
\end{align}
Locally,~$\sN_\rho$ is the negative space of the Hessian~$\Hess(J_\rho)$ at~$(0,0)$.
In principle, since the dimension of~$\sN_\rho$ changes when~$\rho$ runs across an eigenvalue, this would imply that there is a bifurcation phenomena occurring when~$\rho$ is close to the eigenvalues of~$\D_g$, as discussed in~\cite{jevnikar2020existence-sphere}. 
Here we will show that there are solutions for all~$\rho\in (0,\infty)\setminus \Spect(\D_g)$.

For~$\tau>0$ consider the following cone
\begin{align}
 \cC_\tau(\sN_\rho)
 =\braces{ (u,\psi)\in N_\rho 
 \mid  
 \|u\|_{H^1}^2 + \|\psi^-\|^2_{H^{1/2}} + \|\psi^+_a\|^2_{H^{1/2}} 
 < \tau \parenthesis{ \|\psi^+_b\|^2_{H^{1/2}} +\|\psi^0\|^2_{H^{1/2}} } }, 
\end{align}
which increases with respect to~$\tau$. 
In a neighborhood of~$(0,0)$ in~$N_\rho$ but outside the cone~$\cC_\tau(\sN_\rho)$, the functional~$J_\rho$ takes non-negative values. 
More precisely, we have 
\begin{lemma}\label{lemma:coercivity in linking}
 There exist constants~$\tau>1$,~$0<r_0<1$ and~$C>0$ such that,
 \begin{align}
  J_\rho(u,\psi)\ge C(\|u\|^2_{H^1}+\|\psi\|^2_{H^{1/2}}), 
  \qquad 
  \forall 
  (u,\psi)\in (N_\rho\cap B_{r_0}(0,0)) \setminus \cC_\tau(\sN_\rho).
 \end{align}
\end{lemma}
\begin{proof}
 For~$(u,\psi)\in N_\rho$, we take~$R=1$ in~\eqref{eq:control of psi-} (recall that~$R^2$ is an upper bound of~$\|u\|^2_{H^1}+\|\psi\|^2_{H^{1/2}}$) to get 
 \begin{align}
  \|\psi^-\|^2_{H^{1/2}} \le C\rho^2 \parenthesis{
   \|\psi^+_b\|^2_{H^{1/2}} + \|\psi^+_a\|^2_{H^{1/2}}
   +\|\psi^0\|^2_{H^{1/2}}}. 
 \end{align}
 If, in addition,~$(u,\psi)\notin \cC_\tau(\sN_\rho)$, namely
 \begin{align}
  \|u\|_{H^1}^2 + \|\psi^-\|^2_{H^{1/2}} + \|\psi^+_a\|^2_{H^{1/2}} 
  \ge \tau \parenthesis{ 
  \|\psi^+_b\|^2_{H^{1/2}} + \|\psi^0\|^2_{H^{1/2}} },
 \end{align}
 then 
 \begin{align}\label{eq:control of psi-b}
  (\tau-C\rho^2) \parenthesis{ 
  \|\psi^+_b\|^2_{H^{1/2}} + \|\psi^0\|^2_{H^{1/2}} }
 \le \|u\|^2_{H^1} + (1+C\rho^2)\|\psi^+_a\|^2_{H^{1/2}}. 
\end{align}
Consequently, 
\begin{align}\label{eq:control of psi}
 \|\psi\|_{H^{1/2}}^2 
 =&\|\psi^-\|^2_{H^{1/2}} +\|\psi^0\|^2_{H^{1/2}}
    +\|\psi^+_b\|^2_{H^{1/2}}+  \|\psi^+_a\|^2_{H^{1/2}} \\ 
 \le&  (1+C\rho^2)\parenthesis{ \|\psi^0\|^2_{H^{1/2}} + \|\psi^+_b\|^2_{H^{1/2}} +\|\psi^+_a\|^2_{H^{1/2}} } \\
 \le& (1+C\rho^2)\parenthesis{\frac{1}{\tau-C\rho^2} \big(\|u\|^2_{H^1} + (1+C\rho^2)\|\psi^+_a\|^2_{H^{1/2}} \big) + \|\psi^+_a\|^2_{H^{1/2}} } \\
 \le& \frac{1+C\rho^2}{\tau-C\rho^2}\|u\|^2_{H^1}
    +(1+C\rho^2)\frac{\tau+1}{\tau-C\rho^2}\|\psi^+_a\|^2_{H^{1/2}}. 
\end{align}
In particular, 
\begin{align}\label{eq:equivalent norms}
 \|u\|^2_{H^1} + \|\psi\|^2_{H^{1/2}}
 \le &\frac{\tau+1}{\tau-C\rho^2} \parenthesis{\|u\|^2_{H^1} + (1+C\rho^2)\|\psi^+_a\|^2_{H^{1/2}}},
\end{align}
which means that~$(\|u\|_{H^1}^2+ \|\psi^+_a\|^2_{H^{1/2}} )^{1/2}$ defines a locally equivalent norm. 
Note that for~$\tau$ sufficiently large, the coefficient~$(\tau+1)/(\tau-C\rho^2)$ is close to~$1$. 

Consider the spinorial part
\begin{align}
 \int_M \left( 8\Abracket{\D\psi,\psi} \right. 
 & \left. -8\rho\cosh(u)|\psi|^2 \right) \dv_g 
  =\int_M 8\Abracket{(\D-\rho\cosh(u))\psi,
    \psi^0+\psi^+_b +\psi^+_a}\dv_g \\
 =&\int_M 8\Abracket{(\D-\rho)\psi, \psi^0 + \psi^+_b + \psi^+_a }\dv_g \\
  & +\int_M 8\rho\Abracket{(1-\cosh(u))\psi,\psi^0+ \psi^+_b +\psi^+_a}\dv_g \\
 =&\int_M \left( 8\Abracket{(\D-\rho)(\psi^0+\psi^+_b),(\psi^0+\psi^+_b) } 
         +8\Abracket{(\D-\rho)\psi^+_a, \psi^+_a} \right) \dv_g \\
  & +\int_M 8\rho\Abracket{(1-\cosh(u))\psi, \psi^0+ \psi^+_b +\psi^+_a}\dv_g. 
\end{align}
According to the decomposition~\eqref{eq:decomposition of spinor}, each of the above summand can be estimated as follows:
\begin{align}
 \int_M 8\Abracket{(\D-\rho)\psi^+_a,\psi^+_a}  \dv_g 
 \ge \frac{\lambda_{k+1}-\rho}{\lambda_{k+1}+1}\|\psi^+_a\|^2_{H^{1/2}}, 
\end{align}
\begin{align}
 \int_M 8\Abracket{(\D-\rho)(\psi^0+\psi^+_b),(\psi^0+\psi^+_b ) } \dv_g 
 \ge -\rho\parenthesis{\|\psi^+_b\|^2_{H^{1/2}}+\|\psi^0\|^2_{H^{1/2}} }, 
\end{align}
and using~\eqref{eq:control of psi},
\begin{align}
 \Big|\int_M 8\rho &\Abracket{(1-\cosh(u))\psi,
  \psi^0+ \psi^+_b +\psi^+_a}\dv_g \Big|\\
 \le &   C\rho \|u\|_{H^1}\|\psi\|_{H^{1/2}}
      \parenthesis{\|\psi^0\|_{H^{1/2}}+\|\psi^+_b\|_{H^{1/2}} +\|\psi^+_a\|_{H^{1/2}}} \\
 \le & C\rho\|u\|_{H^1}\|\psi\|^2_{H^{1/2}} 
 \le C\rho\frac{1+C\rho^2}{\tau-C\rho^2}\|u\|^3_{H^1}
      +C\rho(1+C\rho^2)\frac{\tau+1}{\tau-C\rho^2}\|u\|_{H^1}\|\psi^+_a\|^2_{H^{1/2}}. 
\end{align}
Thus, for~$(u,\psi)\in (N_\rho\cap B_{r_0}(0,0)) \setminus \cC_\tau(\sN_\rho)$
\begin{align}
 J_\rho(u,\psi)
 \ge&\; C(\rho)\|u\|_{H^1}^2 
 +\frac{\lambda_{k+1}-\rho}{\lambda_{k+1}+1}\|\psi^+_a\|^2_{H^{1/2}} 
 -\rho\|\psi^+_b\|^2_{H^{1/2}}-\rho\|\psi^0\|^2_{H^{1/2}}
  - C\rho\|u\|_{H^1}\|\psi\|^2_{H^{1/2}}  \\
 \ge &\; C(\rho)\|u\|_{H^1}^2 
 +\frac{\lambda_{k+1}-\rho}{\lambda_{k+1}+1}\|\psi^+_a\|^2_{H^{1/2}}  
 -\frac{\rho}{\tau-C\rho^2} \parenthesis{\|u\|^2_{H^1} + (1+C\rho^2)\|\psi^+_a\|^2_{H^{1/2}} }\\
 &\; -C\frac{1+C\rho^2}{\tau-C\rho^2}\|u\|^3_{H^1}
      -C(1+C\rho^2)\frac{\tau+1}{\tau-C\rho^2}\|u\|_{H^1}\|\psi^+_a\|^2_{H^{1/2}} \\
 \ge&\; \parenthesis{C(\rho)-\frac{\rho}{\tau-C\rho^2} - C\frac{1+C\rho^2}{\tau-C\rho^2}r_0 } \|u\|^2_{H^1}  \\
 &\; +\parenthesis{ \frac{\lambda_k-\rho}{\lambda_{k+1}+1}- \frac{\rho(1+C\rho^2)}{\tau- C\rho^2} - C(1+ C\rho^2)\frac{\tau+1}{\tau-C\rho^2}r_0 } \|\psi^+_a\|^2_{H^{1/2}}.
\end{align}
It is now clear that we can choose positive constants~$r_0$ small and then~$\tau$ large to make the two coefficients above positive, say, no less than~$C$:
\begin{align}
 J_\rho(u,\psi)
 \ge C\parenthesis{\|u\|^2_{H^1}+\|\psi^+_a\|^2_{H^{1/2}}}.
\end{align}
The conclusion follows from the equivalence of norms in~\eqref{eq:equivalent norms}. 
\end{proof}

For~$r_0$ chosen as in Lemma~\ref{lemma:coercivity in linking}, consider the set 
\begin{align}
 L_1\coloneqq 
 (\p B_{r_0}(0,0)\cap N_\rho)\setminus \cC_\tau(\sN_\rho). 
\end{align}
This is nonempty since~$(0, r_0\Psi_{k+1}) \in L_1$ and it is homeomorphic to a collar of~$\p B_{r_0}(0,0) \cap \sN_\rho^\perp$. 

Next we will find a subset~$L_2$ which links with~$L_1$ and on which the functional attains non-positive values. 
The construction is similar to the one in~\cite{jevnikar2020existence}; we carry it out here for completeness. 
Consider the finite dimensional ball 
\begin{align}
 B_R(\sN_\rho) = B_R(0,0)\cap \sN_\rho 
 =\{(0,\phi) \; | \; \phi=\phi^0+\phi^+_b\; \mbox{ and } \|\phi\|^2_{H^{1/2}}\le R^2\}. 
\end{align}
On the ball~$B_R(\sN_\rho)$ we have 
\begin{align}
 J_\rho(u,\phi)
 =& \int_M \Abracket{(\D-\rho)(\phi^0+\phi^+_b),(\phi^0+\phi^+_b) }\dv_g 
 \le -\frac{\rho-\lambda_k}{\lambda_k+1} \|\phi\|^2_{H^{1/2}} \le 0. 
\end{align}
For any~$(0,\phi)\in \p B_R(\sN_\rho)$, we will construct a piecewise smooth curve starting from~$(0,\phi)$ and coming back to~$(0,-\phi)\in \p B_R(\sN_\rho)$ (and a line segment in~$B_R(\sN_\rho)$ will help to form a closed curve). 
Let 
\begin{align}
 \sigma_1\colon [0,T]\to N_\rho, \qquad 
 \sigma_1(t)\coloneqq (t,\phi + A t \Psi_{k+1}).
\end{align}
Note that~$\sigma_1(0)=(0,\phi)$ and 
\begin{align}
 J_\rho(\sigma_1(t))
 =&\int_M \left[ 4\rho^2 \sinh(t)^2
 + 8\Abracket{(\D-\rho\cosh(t))\phi,\phi} 
   +  8A^2t^2\Abracket{(\D-\rho\cosh(t))\Psi_{k+1},\Psi_{k+1}} \right] \dv_g \\ 
 =& 4\rho^2 \Vol(M) \sinh(t)^2 
  + \int_M 8\Abracket{(\D-\rho\cosh(t))\phi,\phi}\dv_g
  + 8(\lambda_{k+1}-\rho\cosh(t))A^2 t^2 \\
 \le & 4\rho^2\Vol(M)\sinh(t)^2 
        -\frac{\rho-\lambda_k}{\lambda_k+1} R^2 
        +8(\lambda_{k+1}-\rho\cosh(t))A^2 t^2.
\end{align}
The constants~$T$, $A$ and~$R$ are determined in the following order: 
\begin{enumerate}
 \item[(i)] find first a large $T$ such that~$\lambda_{k+1}-\rho\cosh(T)>1$; 
 \item[(ii)]  fix next~$A$ such that 
            \begin{align}
             4\rho^2\Vol(M)\sinh(T)^2 -8A^2 T^2(\rho\cosh(T)-\lambda_{k+1})<0;
            \end{align}
 \item[(iii)] choose then~$R$ large such that 
            \begin{align}
             \frac{\rho-\lambda_k}{\lambda_k+1} R^2 
             >\max_{t\in [0,T]} 4\rho^2\Vol(M)\sinh(t)^2
             +8(\lambda_{k+1}-\rho\cosh(t))A^2 t^2, 
            \end{align}
            which guarantees that~$J_\rho(\sigma_1(t))<0$ for~$t\in [0,T]$. 
\end{enumerate}
After~$T$,~$A$ and~$R$ are determined,  we join~$(T,\phi+ AT\Psi_{k+1})$ to~$(T, -\phi+AT\Psi_{k+1})$ via a linear segment:
\begin{align}
 \sigma_2\colon [0,1] \to N_\rho,\qquad 
 \sigma_2(s)\coloneqq  (T, (1-2s)\phi + AT\Psi_{k+1}),
\end{align}
and turn to follow the path
\begin{align}
 \sigma_3\colon [0,T]\to N_\rho, \qquad 
 \sigma_3(t)\coloneqq  (T-t, -\phi + A(T-t)\Psi_{k+1} )
\end{align}
to arrive at~$(0,-\phi)\in \p B_R(\sN_\rho)$, as desired. 
The above choice of the constants ensures that~$J_\rho<0$ along the path~$\sigma_1\ast\sigma_2\ast\sigma_3$ (the sum of the paths $\sigma_i$, $i = 1, 2, 3$).

If we write 
\begin{align}
 \sigma_4\colon [0,1] \to N_\rho, \qquad 
 \sigma_4(t)\coloneqq (0, (-1+2s)\phi), 
\end{align}
which traces the diameter connecting~$(0,-\phi)$ and~$(0,\phi)$, then the path~$\sigma_1\ast\sigma_2\ast\sigma_3\ast\sigma_4$ is a closed piecewise smooth curve based at~$(0,\phi)\in \p B_R(\sN_\rho)$. 
Letting~$(0,\phi)$ run throughout~$\p B_R(\sN_\rho)$ and collecting all such curves, we see that they knit the boundary of a solid cylinder segment~$\cD$ where
\begin{align}
 \cD=\{(t,\phi + t\Psi_{k+1})\in N_\rho \; \mid \; 0\le t\le T, \; \phi\in \parenthesis{ H^{\frac{1}{2},0} +H^{\frac{1}{2},+}_b}(\Sigma M),\;  \|\phi\|_{H^{1/2}}^2 \le R^2\}.
\end{align}
Let~$L_2\coloneqq \p \cD$, which links~$L_1$: by shrinking~$L_2$ a little we can get a similar picture as if we were in a Hilbert space, where we can rely on the classical theory as in~\cite[Section 8.3]{ambrosetti2007nonlinear} to see that they actually link each other. 

We can then find min-max solutions at the linking level as follows. 
Let~$\Gamma$ be the space of continuous maps~$\alpha\colon \cD\to N$ which fix the boundary~$\p \cD$. 
Since~$\id_\cD$ is such a map,~$\Gamma\neq \emptyset$.
The linking level is defined as 
\begin{align}
 c_1\coloneqq \inf_{\alpha\in\Gamma } \max_{x\in \Gamma} J_\rho(\alpha(x)). 
\end{align}
Note that Lemma~\ref{lemma:coercivity in linking} implies
\begin{align}
 c_1\ge C r_0^2 >0.
\end{align}
The standard theory (see e.g.~\cite{ambrosetti1973dual, ambrosetti2007nonlinear}) now applies to give a nontrivial critical point of~$J_\rho|_{N_\rho}$ at the level~$c$, which are the nontrivial solutions of~$J_\rho$ that we are looking for. 

\

\section{A multiplicity result} \label{sec:mult}

We  obtained a nontrivial min-max solution by the  mountain pass or linking methods, and let us denote it by~$(u^{(1)}, \psi^{(1)})$. 
Since the functional is invariant under the action of~$\mathbb{Z}_2\times \mathcal{J}$, we get a three-dimensional orbit
\begin{align}
 \mathscr{U}_1\coloneqq \{(\sigma u^{(1)}, \mathbf{j}(\psi^{(1)}))\mid \sigma \in \mathbb{Z}_2 = \{\pm1\}, \; \mathbf{j}\in\mathcal{J}\}
\end{align}
which consists of two components, each homeomorphic to a 3-sphere.
All the elements in~$\mathscr{U}_1$ are solutions of~\eqref{eq:SShG}, but they are geometrically the same.
To be consistent, let us denote the trivial solution by~$(u^{(0)},\psi^{(0)})\equiv (0,0)$ and write its (trivial) orbit as~$\mathscr{U}_0=\{(u^{(0)},\psi^{(0)})\}$. 
Our aim here is to prove Theorem~\ref{thm:main thm2} and find a (family of) geometrically distinct solutions from those in~$\mathscr{U}_0\cup\mathscr{U}_1$, by exploiting the~$\mathbb{Z}_2$-symmetry on the function components, i.e. the evenness of the functional in the~$u$-component, and using the \emph{fountain theorem}, see e.g.~\cite{willem1997minimax}. In this section, all the~$\mathbb{Z}_2$-symmetries are referred to the~$u$-component only.  

\

To begin, we first construct a family of \emph{sweepout} functions connecting the constant functions~$+1$ and~$-1$, see Figure~\ref{fig}. 
Embed the surface into some Euclidean space~$\R^{m}$ with the last coordinate function~$x^m$ being a Morse height function such that~$\min_M x^m=0$ and~$\max_M x^m=\pi$.
Let~$f\colon\R^m \to \R$ be a~$2\pi$-periodic function which takes the value~$+1$ on~$\{x^m\in [2k\pi, (2k+1)\pi)\mid k\in\bZ\}$ and~$-1$ otherwise. 
Consider the function~$\tilde{\chi}\colon\R\times M\to \R$ defined by 
\begin{align}
 \tilde{\chi}(\theta,x)= f(x+ (0,\cdots, \theta)), 
\end{align}
which is~$2\pi$-periodic in~$\theta$. 
Note that, aside from~$\theta$, the function~$\tilde{\chi}$ depends only on the last coordinate. 

\begin{lemma}\label{lemma:sweepout}
 Given~$\eps>0$, there exists a smooth function~$\chi\colon \sph^1\times M \to [-1,1]$, where~$\sph^1=\faktor{ \R}{2\pi\bZ}$, such that 
 \begin{enumerate}
  \item[(i)] $\chi(0, \cdot)\equiv 1$,
  \item[(ii)] $\chi(\theta+\pi, \cdot )= -\chi(\theta,\cdot)$, ~$\forall \theta\in \sph^1=\faktor{\R}{2\pi\bZ}$, 
  \item[(iii)] $\Vol(\{-1<\chi(\theta, \cdot)<1\})<\eps$.
 \end{enumerate}
\end{lemma}
Note that (i) and (ii) yield~$\chi(\pi,\cdot)=-1$, while (iii) implies 
\begin{align}
 \int_M \big| |\chi(\theta,x)|- 1\big|^p \dv_g < \eps, \qquad \forall\;  \theta\in \faktor{\R}{2\pi\bZ} \quad \mbox{ and }\quad \forall\; p\in [1,\infty). 
\end{align}
\begin{proof}[Proof of Lemma~\ref{lemma:sweepout}]
 This can be achieved by a standard mollifying procedure. 
 More precisely, let~$\eta\in C^\infty_c( (-\delta,\delta), \R)$ be a smooth bump function with integral~$1$, and take the convolution
 \begin{align}
  \chi(\theta, x) = \int_{\R} \eta(y)\tilde{\chi}(\theta, x^1,\cdots, x^{m-1}, x^m-y)\dd y. 
 \end{align}
 Then~$\chi$ is smooth. 
 Since~$\tilde{\chi}$ satisfies (i) and (ii) a.e., so does~$\chi$. See Figure~\ref{fig} for a schematic picture of the map~$\chi$. 
 As for (iii): the embedding of~$M$ into~$\R^m$ can be taken so that for any~$\theta\in \R$, the subset 
 \begin{align}
  \{x\in M \mid \theta-\delta \le x^m \le \theta+\delta\}
 \end{align}
 has volume less than~$\eps$. 
 Then it suffices to notice that 
 \begin{align}
  \{-1<\chi(\theta,\cdot)<1\} 
  = \{ |\chi(\theta,\cdot) |\neq 1 \} 
  \subset \{\theta-\delta \le x^m \le \theta+\delta\}. 
 \end{align}


\begin{figure}[h]  
\centering
\includegraphics[width=0.7\linewidth]{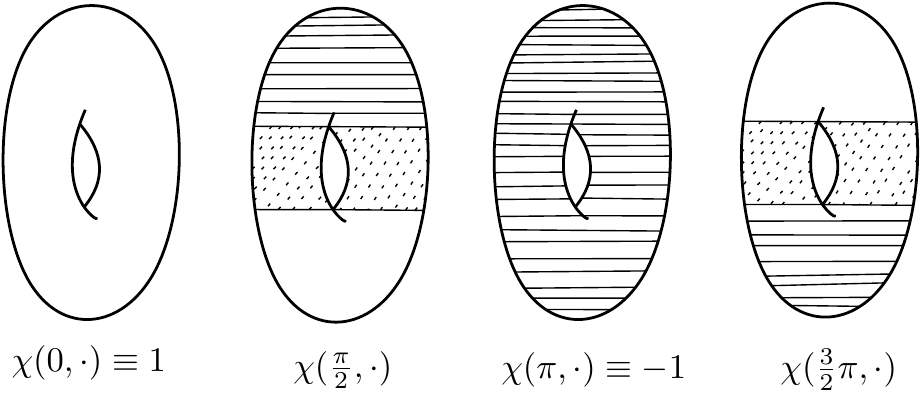}
\caption{The function takes value~$+1$ in the white part, and value~$-1$ on the lined part, while~$-1<\chi<1$ in the dotted part. }
\label{fig}
\end{figure}


\end{proof}

Now for an arbitrary function~$u \in C^\infty(M)$, multiplying by~$\chi$ gives a family~$\{u_\theta\in C^\infty(M) \mid \theta\in \sph^1=\faktor{\R}{2\pi\bZ}\}$ where~$u_\theta(x)= \chi(\theta,x)u(x)$  satisfies~$u_0=u$,~$u_\pi= - u$, and more generally
\begin{align}
 u_{\theta+\pi}(x)= - u_\theta(x). 
\end{align}
Moreover, 
\begin{align}
 \| |u_\theta|-|u|\|_{L^p} <\eps^{\frac{1}{2p}}\|u\|_{L^{2p}}, \qquad \forall \; p\in [1,\infty). 
\end{align}

\ 

To make the argument clearer, we first discuss the case~$h=0$,~$0<\rho<\lambda_1$, and then consider the general case, as in the previous section. 
We will explain the first case in detail and will be rather sketchy in the second case. 

\

\noindent\textbf{Case 1. $h=0$ and~$0<\rho<\lambda_1$.}
In Section~\ref{sec:mountain pass} we have seen that~$(\bar{u}, s\Psi_1) \in N_\rho$ and~$J_\rho(\bar{u}, s\Psi_1)<0$, where~$\bar{u}\in \R$ is large (in particular~$\bar{u}>0$), and~$s\in \R$ is also large enough. 
For each~$\theta\in \faktor{\R}{2\pi\bZ}$, set
\begin{align}
 u_\theta (x) \coloneqq \chi(\theta,x) \bar{u}, \qquad x\in M. 
\end{align}
Next, we need to find~$\psi_\theta$ such that~$\psi_0=s\Psi_1$,~$\psi_{\theta+\pi}= \psi_\theta$, and 
\begin{align}
 (u_\theta,\psi_\theta)\in N_\rho, & & 
 J_\rho(u_\theta,\psi_\theta)\le 0, & & \forall \; \theta\in \faktor{\R}{2\pi\bZ}. 
\end{align}
Recall that for any~$(u,\psi)$ the map 
\begin{align}
  H^{\frac{1}{2},-}(\Sigma M) &\to H^{\frac{1}{2},-}(\Sigma M), \\
 \phi &\mapsto P^- (1+|\D|)^{-1} (\D(\psi+\phi)-\rho \cosh(u) (\psi+\phi))
\end{align}
is a local isomorphism since its differential is an isomorphism, see Section~\ref{sec:setting}. 
For~$u_\theta$ and~$s\in \R$ chosen as above, consider the map 
\begin{align}
 F\colon & \sph^1 \times  H^{\frac{1}{2},-}(\Sigma M) \to H^{\frac{1}{2},-}(\Sigma M), \\
  & F(\theta, \phi)= P^-(1+|\D|)^{-1}(\D(s\Psi_1+\phi)-\rho\cosh(u_\theta)(s\Psi_1+\phi)). 
\end{align}
We have~$F(0,0)=0$ since~$(\bar{u}, s\Psi_1)\in N_\rho$, and~$D_2 F(\theta,\phi)$ is an ismorphism. 
By the Implicit Function Theorem, there exists a local neighborhood~$(-\delta, \delta) \subset \sph^1= \faktor{\R}{2\pi\bZ}$ of~$\theta =0$ and a smooth function
\begin{align}
 (-\delta,\delta) \to H^{\frac{1}{2},-}(\Sigma M), \qquad \theta\mapsto \phi_\theta
\end{align}
satisfying~$F(\theta, \phi_\theta)=0$ for each~$\theta\in (-\delta,\delta)$ and~$\phi_0=0$. 
Putting~$\psi_\theta\equiv s\Psi_1+\phi_\theta$, we have~$(u_\theta,\psi_\theta)\in N_\rho$, ~$|\theta|<\delta$. 
Moreover, since~$D_2F(\theta,\phi)$ has bounded operator norm, also bounded away from zero, we can choose a uniform~$\delta>0$ such that for each~$\theta\in \sph^1$ and~$(u_\theta,\psi_\theta)\in N_\rho$, ~$\phi_\theta$ can be defined on the~$\delta$-neighborhood of~$\theta$.
Since~$\phi_\theta$ is smooth in~$\theta\in\sph^1$ and~$\phi_0=0$, we have
\begin{align}
 \|\phi_\theta\|_{H^{1/2}} \le C
\end{align}
for some constant~$C=C(u_\theta)$, which is independent of~$s>0$ since~$D_2 F(u_\theta,\psi_\theta)$ is independent of~$s$. 

Using the continuity method we get a well-defined family~$\{(u_\theta,\psi_\theta)\mid {\theta\in\sph^1} \}$ lying
in~$N_\rho$. 
Note that, since~$(-\bar{u}, s\Psi_1)\in N_\rho$ and~$u_\pi= -\bar{u}$, we must have~$\psi_\pi= s\Psi_1$ and it follows from  uniqueness that~$\psi_{\theta+\pi}=-\psi_\theta$, for any~$\theta\in\sph^1$. 

Now we get a set~$S\equiv \{(u_\theta,\psi_\theta)\mid {\theta\in\sph^1} \} \cong \sph^1$ in~$N_\rho$ which is homeomorphic to the circle~$\sph^1$.  
For each~$\theta$, 
\begin{align}
 J_\rho(u_\theta,\psi_\theta)
 =&\int_M \left[ |\nabla u_\theta|^2 + 4\rho^2\sinh(u_\theta)^2 
 +8\Abracket{\D\psi_\theta -\rho\cosh(u_\theta)\psi_\theta,\psi_\theta} \right] \dv_g \\
 =& \int_M \left[ \bar{u}^2|\nabla\chi(\theta,\cdot)|^2 + 4\rho^2\sinh(u_\theta)^2 
  +8 \Abracket{(\D-\rho\cosh(u_\theta))\psi_\theta, s\Psi_1} \right] \dv_g  \\
 =& \bar{u}^2\int_M \ |\nabla \chi(\theta,\cdot)|^2 \dv_g  
   +4\rho^2 \int_M \sinh(u_\theta)^2\dv_g 
   +8s^2\int_M \Abracket{(\D-\rho\cosh(\bar{u}))\Psi_1,\Psi_1} \dv_g \\
   &+8\rho\int_{\{ |u_\theta| \neq \bar{u}\}}\parenthesis{\cosh(\bar{u})-\cosh(u_\theta)}\Abracket{s\Psi_1+\phi_\theta, s\Psi_1}\dv_g  \\
 \le &\bar{u}^2 \int_M |\nabla\chi(\theta,\cdot)|^2\dv_g 
      + 4\rho^2 \sinh(\bar{u})^2 \Vol(M,g)
      - 8s^2 (\rho\cosh(\bar{u})-\lambda_1)\\ 
     & +8\rho (Cs^2+Cs)\cosh(\bar{u}) \eps^{1/2}, 
\end{align}
where in the last step we  used the inequality  
\begin{align}
 8\rho\int_{\{| u_\theta|\neq \bar{u}\} }& \parenthesis{\cosh(\bar{u})-\cosh(u_\theta)}\Abracket{s\Psi_1+\phi_\theta, s\Psi_1}\dv_g\\
 &\le 8\rho\cosh(\bar{u})\int_{\{ |u_\theta|\neq \bar{u}\}} \left[ s^2|\Psi_1|^2 + s\Abracket{\phi_\theta, \Psi_1} \right] \dv_g \\
 &\le 8\rho\cosh(\bar{u})\Vol(\{|u_\theta|\neq \bar{u}\})^{1/2} \braces{ s^2 \|\Psi_1\|_{L^4}^2
 +s \|\phi_\theta\|_{L^4} \|\Psi_1\|_{L^4}
 }.
\end{align}
Hence,
\begin{align}
 J_\rho(u_\theta,\psi_\theta)
 \le &\bar{u}^2 \int_M |\nabla\chi(\theta,\cdot)|^2\dv_g 
      + 4\rho^2 \sinh(\bar{u})^2 \Vol(M,g) \\
     & -8s^2\parenthesis{\rho\cosh(\bar{u})-\lambda_1-2C\rho\cosh(\bar{u})\eps^{1/2} }.  
\end{align}

Therefore, we could choose~$\bar{u}\gg 1$,~$0<\eps\ll 1$ and~$s\gg 1$ to achieve 
\begin{align}
 J_\rho(u_\theta,\psi_\theta)<0, \qquad \forall \; \theta\in\sph^1. 
\end{align}
Note also that~$S$ is~$\bZ_2$-invariant, where~$\bZ_2=\{\pm1\}$ only acts on the~$u$-components. 

Recall that~$N_\rho$ is contractible and~$\bZ_2$-invariant, in particular it is simply connected. 
Hence we can find a $\bZ_2$-invariant two-dimensional \emph{topological} disk~$\bB\subset N_\rho$ (here we mean a subset of~$N_\rho$ which is diffeomorphic to the standard unit disk~$B_1^2(0)\subset \R^2$) with boundary~$S$.
For example, one can consider a~$\bZ_2$-invariant disk in the space~$H^1(M)$ with boundary~$\{(u_\theta,0)\mid \theta\in\sph^1\}$ and then connect~$(u_\theta,0)$ to~$(u_\theta,\psi_\theta)\in S$ within the fiber~$N_{\rho, u_\theta}$ for each~$\theta\in\sph^2$. 
Therefore, we get a~$\bZ_2$-equivariant map 
\begin{align}
 w\colon B^2_1(0)\to \bB\subset N_\rho.
\end{align}

Consider the collection 
\begin{align}
 \Gamma_2 \coloneqq\{\alpha \in C(B^2_1(0), N_\rho)| \; \alpha \mbox{ is } \bZ_2\mbox{-equivariant }, \; \alpha|_{\p B^2_1(0)}= w|_{\p B^2_1(0)}\}.
\end{align}
This is nonempty since~$w \in\Gamma_2$. 
Then the value defined by
\begin{align}
 c_2 \coloneqq \inf_{\alpha\in\Gamma_2} \max_{ z\in B_1^2(0)} J_\rho(\alpha(z))
\end{align}
is again a critical level, see e.g.~\cite{willem1997minimax}. 
Note that~$c_2$ is finite, since the mountain pass geometry guarantees that it is positive, and~$\max J_\rho\circ w$ is finite.  

Note that for~$z_0= (1,0)\in \p B^2_1(0)\subset \R^2$, ~$w(z_0)= (\bar{u},s\Psi_1)\in N_\rho$. 
If~$\widetilde{\alpha}$ is a curve connecting~$(0,0)$ and~$(\bar{u}, s\Psi_1)$ within~$N_\rho$, then we can use the above procedure to construct a~$\bZ_2$-equivariant~$\alpha\in\Gamma_2$ such that~$\alpha|_{[0,1]\times\{0\}}$ coincides with~$\tilde{\alpha}$. 
Thus~$\max_{B^2_1(0)}\alpha\ge \max_{[0,1]}\widetilde{\alpha}$ and~$c_2\ge c_1$.

If~$c_2>c_1$,  we get another solution~$(u^{(2)},\psi^{(2)})$ which is geometrically distinct from $(u^{(1)},\psi^{(1)})$ since they lie in different energy levels. 
The case~$c_2=c_1$ can be dealt with as follows. 
Consider the subspace of~$H^1(M)$ which is perpendicular to~$u^{(1)}$, and 
\begin{align}
 \widetilde{N}_\rho\coloneqq\braces{ 
  (u,\psi)\in N_\rho \mid \Abracket{ u, u^{(1)}}_{H^1}=0
 }.   
\end{align}
This is a submanifold of~$N_\rho$, which again admits a local mountain pass geometry: locally,~$J_\rho$ grows up around~$(0,0)$ and we can find~$\theta_0\in (0,\pi)$ such that~$(u_{\theta_0}, \psi_{\theta_0}) \in \widetilde{N}_\rho$ with~$J_\rho(u_\theta,\psi_\theta)<0$. 
Indeed, if~$\Abracket{u^{(1)}, u_\theta}_{H^1}=0$ for all~$\theta\in\faktor{\R}{2\pi\bZ}$, then we may take~$\theta_0=0$, so~$(u_0,\psi_0)$ meets our requirement; otherwise there is~$\theta_1$ such that
\begin{align}
 \Abracket{u^{(1)}, u_{\theta_1}}_{H^1}
 =-\Abracket{u^{(1)}, u_{\theta_1+\pi}}_{H^1}>0,  
\end{align}
and the intermediate value theorem implies the existence of~$\theta_0$. 
Applying the mountain pass method we can find a non-zero critical point~$(u^{(2)}, \psi^{(2)})\in\widetilde{N}_\rho$ with critical level~$c'_1$. 
The curves connecting~$(0,0)$ to~$(u_{\pi/2}, \psi_{\pi/2})$ in~$\widetilde{N}_\rho$ can also be extended to a map~$\alpha\in \Gamma_2$: thus~$c'_1\le c_2$. 
However, by definition of~$c_1$, we have~$c'_1\ge c_1$. 
Therefore~$c'_1= c_1$ and~$(u^{(2)}, \psi^{(2)})$ is a true critical point at level~$c_1$ and~$\Abracket{u^{(1)}, u^{(2)}}_{H^1}=0$. 

In any case, we get another family~$\mathscr{U}_2$ of min-max solutions.
 
\

\noindent\textbf{Case 2: $h>0$ or~$\rho>\lambda_1$.}

In Section~\ref{sec:linking} we  found a solid cylinder segment~$\cD\subset N_\rho$ such that~$J_\rho|_{\p \cD}\le 0$.
Let~$K$ denote the dimension of~$\sN_\rho$, so that the solid cylinder~$\cD$ is homeomorphic to~$B^K_1(0)\times [0,1]$. 

Note that the elements on~$\p\cD$ have constant~$u$-components. 
Therefore, we may use the above procedure to obtain a~$\bZ_2$-invariant~$(K+2)$-dimensional set~$\bar{\cD}$, homeomorphic to~$B^K_1(0)\times B^2_1(0)$, whose boundary is~$\p B^K_1(0)\times B^2_1(0) \cup B^K_1(0)\times \p B^2_1(0)$. 
Since this is a compact set, a suitable choice of the parameters~$A,T,R$ and~$\eps$ will guarantee that~$J_\rho |_{\p\bar{\cD}}\le 0$. 
This is done in the same fashion as above and we will omit the details.

Finally, define
\begin{align}
 \Gamma_2 \coloneqq \{\alpha\in C(\bar{\cD}, N_\rho) : 
 \alpha \mbox{ is } \bZ_2\mbox{-equivariant and } \; \alpha_{\p \bar{\cD} }=\id_{\p\bar{\cD}}\}, 
\end{align}
which is nonempty since~$\id_{\bar{\cD}}\in\Gamma_2$.
The value
\begin{align}
 c_2 \coloneqq \inf_{\alpha\in\Gamma_2} \max_{(u,\psi)\in \bar{\cD}}  J_\rho(u,\psi)
\end{align}
is then a critical level, which gives rise to a second nontrivial solution as in the first case. 

\

We conclude this section with the following remark.
\begin{rmk}\label{r:last}
 The~$\bZ_2$-symmetry of the problem gives us the second family of solutions, which was not present in the super Liouville case~\cite{jevnikar2020existence}. 
 One naturally wonders whether there are infinitely many geometrically distinct families of solutions. 
 By taking more general sweepout functions, see e.g. \cite{gasparguaraco2018CV}, one might get another (family of) solutions.
 However, to get infinitely many families seems difficult: we do not control the spinor part very clearly when~$u$ is not constant.  
 Still, we do think there should be infinitely many such families.
 
 Another remark concerns the~$\sph^3$ symmetry (on spinors) of the problem. 
 It would be interesting to exploit this additional symmetry to get extra multiplicity.

\end{rmk}

\end{document}